\documentclass[12pt]{article}

\usepackage[flushleft]{threeparttable}
\usepackage{graphicx}
\usepackage{subcaption}
\usepackage{empheq}

\RequirePackage{amsmath} 
\RequirePackage{amsthm}
\RequirePackage{amsfonts}

\newcommand{\hf}{\frac{1}{2}}
\newcommand{\inti}{\int_{I_i}}
\newcommand{\intj}{\int_{J_j}}

\DeclareMathOperator{\diag}{diag}
\DeclareMathOperator{\pre}{pre}

\newcommand{\vphi}{\varphi} 
\newcommand{\veps}{\varepsilon}

\newcommand{\ninti}{\overset{\sim}{\int_{I_i}}}
\newcommand{\nintI}{\overset{\sim}{\int_{I}}}
\newcommand{\nintj}{\overset{\sim}{\int_{J_j}}}
\newcommand{\nintJ}{\overset{\sim}{\int_{J}}}
\newcommand{\nintO}{\overset{\sim}{\int_{\Omega}}}

\newcommand{\cI}{\mathcal{I}}
\newcommand{\mI}{\mathcal{I}}

\newcommand{\RR}{\mathbb{R}}
\newcommand{\bpsi}{\boldsymbol{\psi}}

\newcommand{\bvphi}{\boldsymbol{\varphi} }
\newcommand{\brho}{\boldsymbol{\rho}}
\newcommand{\bxi}{\boldsymbol{\xi}}

\newcommand{\bs}{\boldsymbol{s}}
\newcommand{\bu}{\boldsymbol{u}}
\newcommand{\bv}{\boldsymbol{v}}
\newcommand{\bx}{\boldsymbol{x}}
\newcommand{\bz}{\boldsymbol{z}}
\newcommand{\bF}{\boldsymbol{F}}
\newcommand{\bV}{\boldsymbol{V}}
\newcommand{\wFu}{\widehat{Fu}}

\newcommand{\wFbu}{\widehat{F\bu}}
\newcommand{\wFbux}{\widehat{F\bu^x}}
\newcommand{\wFbuy}{\widehat{F\bu^y}}
\newcommand{\wbxi}{\widehat{\bxi}}
\newcommand{\barrho}{\bar{\rho}}

\newcommand{\bzeta}{\boldsymbol{\zeta}}
\newcommand{\bseta}{\boldsymbol{\eta}}

\newtheorem{THM}{Theorem}[section]

\newtheorem{REM}{Remark}[section] \theoremstyle{definition}

\newtheorem{Examp}[subsection]{Example}

\allowdisplaybreaks

\begin{document}

\baselineskip=2pc
\vspace*{.10in}
\begin{center}
  {\bf An entropy stable high-order discontinuous Galerkin method \\ for
  cross-diffusion gradient flow systems}
\end{center}
\centerline{
	Zheng Sun\footnote{Department of Mathematics, The Ohio State University,
  Columbus, OH 43210, USA. E-mail: sun.2516@osu.edu.},
  Jos\'{e} A. Carrillo\footnote{Department of Mathematics, Imperial College London, London
SW7 2AZ, UK. E-mail: carrillo@imperial.ac.uk.} and 
  Chi-Wang Shu\footnote{Division of Applied Mathematics, Brown University,
    Providence, RI 02912, USA. E-mail: shu@dam.brown.edu.}
}

\vspace{.2in}

\centerline{\bf Abstract} \bigskip

As an extension of our previous work in \cite{sun2018discontinuous}, we develop a
discontinuous Galerkin method for solving cross-diffusion systems with a formal
gradient flow structure. These systems are associated with non-increasing
entropy functionals. For a class of problems, the positivity
(non-negativity) of solutions is also expected, which is implied by the physical
model and is crucial to the entropy structure. The
semi-discrete numerical scheme we
propose is entropy stable. Furthermore,
the scheme is also compatible with the positivity-preserving procedure in
\cite{zhang2017positivity} in many scenarios. Hence the resulting fully discrete scheme is able to 
produce non-negative solutions. The method can be applied to both
one-dimensional problems and two-dimensional problems on Cartesian meshes.
Numerical examples are given to examine the performance of the method. 
\vfill


\clearpage

\section{Introduction}
\setcounter{equation}{0}
\setcounter{figure}{0}
\setcounter{table}{0}

Cross-diffusion systems are widely used to model multi-species interactions in various applications, such as population dynamics in biological systems \cite{shigesada1979spatial}, chemotactic cell migration
\cite{keller1970initiation}, spread of surfactant on the membrane \cite{jensen1992insoluble}, interacting particle systems with volume exclusion \cite{bruna,berendsen}, and pedestrian dynamics \cite{HRSW}. In many situations, the systems are associated with neither symmetric nor positive-definite diffusion matrices, which not only complicate the mathematical analysis but also hinder the development of numerical methods. Recently, progress has been made in analyzing a class of cross-diffusion systems having the structure of a gradient flow associated with a dissipative entropy (or free energy) functional, see \cite{jungel2015boundedness, jungelzamponi,jungelzamponi2} and references therein. We will exploit this gradient flow structure to develop a high-order stable discontinuous Galerkin (DG) method for these cross-diffusion systems.

Let $\Omega$ be a bounded domain in $\RR^d$. We are interested in solving the
following initial value problem of the cross-diffusion system. 
\begin{equation}\label{eq-gfs}
  \left\{
\begin{aligned}
  \partial_t \brho = \nabla \cdot \left(F(\brho) \nabla \bxi(\brho) \right)
  &:= \sum_{l=1}^d \partial_{x_l} \left(F(\brho) \partial_{x_l} \bxi(\brho) \right),\qquad
  (\bx,t) \in  \Omega \times [0,\infty),\\
  \brho(\bx,0) &= \brho_0(\bx).
\end{aligned}\right.
\end{equation}
Here $\brho = \brho(\bx,t) $ and $\brho= (\rho_1,\cdots,\rho_m)^T$ is a vector-valued
function as well as $\bxi :=  \frac{\delta E}{\delta \brho} =  (\partial_{\rho_1} e,\cdots, \partial_{\rho_m} e)^T$ with $e = e(\brho)$ being a scalar-valued twice differentiable convex function. 
We also assume $F$ to be an $m\times m$ positive-semidefinite $\brho$-dependent matrix, in the
sense that $\bz \cdot F \bz \geq 0$. Note $F$ can be non-symmetric. Equation \eqref{eq-gfs} can be written in divergence form as $\partial_t \brho = \nabla\cdot(A(\brho)\nabla\brho)$, with the matrix $A(\brho)
= F(\brho) D\bxi(\brho) = F(\brho) D^ 2 e(\brho)$ possibly neither symmetric nor positive-definite. 

The system \eqref{eq-gfs} possesses a formal gradient flow structure governed by
the entropy functional 
\begin{equation}\label{eq-ent}
  E =
\int_\Omega e(\brho) dx.
\end{equation}The system can be rewritten as 
$
  \partial_t \brho = \nabla \cdot \left(F(\brho) \nabla \frac{\delta E}{\delta
\brho} \right)$.
One can see at least for classical solutions that
\begin{equation}\label{eq-entdecay}
  \frac{d }{dt}E = \int_\Omega  {\partial}_t\brho \cdot \bxi dx = - \int_\Omega
D\bxi : F D\bxi dx = - \sum_{l=1}^d \int_\Omega
\partial_{x_l} \bxi \cdot F \partial_{x_l}\bxi dx \leq 0\,,
\end{equation}
with the usual notation for the matrix product of square matrices $A:B:=\sum_{i,j} a_{ij} b_{ij}$. The Liapunov functional \eqref{eq-ent} due to \eqref{eq-entdecay} indicates certain well-posedness
of the initial value problem \eqref{eq-gfs}. However, in applications 
with $\brho$ representing non-negative physics quantities, such as 
species densities, mass fractions and water heights, the well-posedness 
usually relies heavily on the positivity of the solution.
For example, with $F(\brho) = \diag(\brho)$, $F(\brho)$ is semi-positive 
definite only when $\brho$ is non-negative. 
Violating the non-negativity may result in a non-decreasing entropy 
in \eqref{eq-entdecay}. Furthermore, in problems that
a logarithm entropy $E = \int_\Omega \sum_{l=1}^m\rho_l (\log 
\rho_l-1) dx$ is considered, the entropy may not even be well-defined 
when $\brho$ admits negative values. 

The entropy structure is crucial to understand various theoretical properties of 
cross-diffusion systems, such as existence, regularity and long time asymptotics of 
weak solutions, see \cite{jungel2015boundedness, jungelzamponi,jungelzamponi2}. It is desirable to design numerical schemes that preserve the entropy decay 
in \eqref{eq-entdecay}, and would also be positivity-preserving if the solution to the continuum equation is
non-negative, due to the concern we have mentioned. 
Various efforts have been spent on numerical methods for scalar problems with a
similar entropy structure, including the mixed finite element method \cite{burger2010mixed}, the finite 
volume method \cite{bessemoulin2012finite,carrillo2015finite}, the direct DG
method \cite{LH,liu2016entropy}, 
optimal mass transportation based methods \cite{CM,JMO,CRW,CDM}, the particle method \cite{carrillo2017numerical} 
and the blob method \cite{craig2016blob, carrillo2017blob}. Some methods can also be
generalized to systems, for example the Poisson-Nerst-Plack system
\cite{liu2017free} and system of interacting species with cross-diffusion
\cite{carrillo2018convergence}. 

In this paper, we extend the DG method in
\cite{sun2018discontinuous} for scalar gradient flows to cross-diffusion
systems. The DG method is a class of finite element methods utilizing
discontinuous piecewise polynomial spaces, which was originally designed
for solving transport equations \cite{reed1973triangular} 
and was then developed for nonlinear
hyperbolic conservation laws \cite{cockburn1991runge1, cockburn1989tvb2,
cockburn1989tvb3,cockburn1990runge4, cockburn1998runge5}. 
The method has also been generalized for problems involving diffusion and higher
order derivatives, for example, the local DG method \cite{bassi1997high, 
cockburn1998local}, the ultra-weak DG method \cite{cheng2008discontinuous} 
and the direct DG method \cite{liu2009direct}. 
The method preserves local conservation, achieves high-order accuracy, is able
to handle
complex geometry and features with good $h$-$p$ adaptivity and high parallel
efficiency. As an effort to incorporate these potential advantages in gradient flow simulations, 
in \cite{sun2018discontinuous}, 
we adopted the technique from \cite{chen2017entropy} 
to combine the local DG methods with Gauss-Lobatto quadrature rule, and 
designed a DG method for scalar gradient flows 
that is entropy stable on the semi-discrete level. This method also 
features weak positivity with Euler forward time stepping method. 
Hence after applying a scaling limiter and the 
strong-stability-preserving Runge-Kutta (SSP-RK) time discretization
\cite{gottlieb2001strong}, the fully 
discrete scheme produces non-negative solutions. This positivity-preserving
procedure is established in 
\cite{zhang2010maximum,zhang2017positivity,srinivasanapositivity}. 

The main idea for handling cross-diffusion systems is to formally
rewrite the problem into decoupled equations and apply our previous
numerical strategy in \cite{sun2018discontinuous} for scalar gradient flows
to the unknown density vector component-wise. In fact, the system \eqref{eq-gfs} can be rewritten as
\[
\partial_t \brho = \nabla \cdot (\diag(\brho) \bv),\qquad \bv = \diag(\brho)^{-1}F(\brho)\bu, \qquad \bu = \nabla \bxi.
\]
In the particular case of one dimension, we are reduced to apply our previous scheme in \cite{sun2018discontinuous} to $\partial_t \rho_l = \partial_x (\rho_l  v_l)$, for $l=1,\dots, m$.
The numerical fluxes are properly chosen to ensure the entropy decay and the
weak positivity. In this approach, the boundedness of $\diag(\brho)^{-1}F(\brho)$
is required for $\bv$ to be well-defined.  This assumption does 
hold in various applications, see examples in Section \ref{sec-tests1}
and Section \ref{sec-tests2}.

The rest of the paper is organized as follows. In Section \ref{sec-1dnum} and
Section \ref{sec-2dnum}, we propose our DG schemes for
one-dimensional and two-dimensional gradient flow systems respectively. Each
section is composed of three parts: notations, semi-discrete scheme and entropy
inequality, and fully discrete scheme as well as positivity-preserving
techniques for producing non-negative solutions (with central and Lax-Friedrichs
fluxes). In Section \ref{sec-tests1} and Section \ref{sec-tests2}, we show several numerical examples to 
examine the performance of the numerical schemes. Finally, concluding remarks are given
in Section \ref{sec-conclusions}. 


\section{Numerical method: one-dimensional case}\label{sec-1dnum}
\setcounter{equation}{0}
\setcounter{figure}{0}
\setcounter{table}{0}
\subsection{Notations}
In this section, we consider the one-dimensional
problem 
\begin{equation}\label{eq-1d}
\left\{\begin{aligned} 
    \partial_t \brho &= \partial_x \left(F\partial_x \bxi\right), x\in
    \Omega\subset \RR,
    t>0.\\
  \brho(x,0) &= \brho^0(x).
\end{aligned}\right.
\end{equation}
For simplicity, the compact support or periodic boundary condition is 
assumed, while it can
also be extended to the zero-flux boundary condition.

Let $I_i = (x_{i-\hf},x_{i+\hf})$ and $I = \cup_{i=1}^N I_i$ be a regular partition of the domain $\Omega$. 
The length of the mesh cell is denoted
by $h_i = x_{i+\hf}-x_{i-\hf}$. We assume $h = \max_{i=1,\ldots,N} h_i$ and $h_i \geq c_{mesh}h$ for
some fixed constant $c_{mesh}$. The numerical solutions are defined in the
tensor product space of discontinuous piecewise polynomials
\[\bV_h = \prod_{l = 1}^mV_h, \qquad V_h = \{v_h(x): v_h|_{I_i} \in P^k(I_i)\}.\]
$P^k(I_i)$ is the space of polynomials of degree $k$ on $I_i$. Functions in
$\bV_h$ ($V_h$)
can be double-valued at cell interfaces. We use $\bv_h^-$ ($v_{h,l}^-$) and
$\bv_h^+$ ($v_{h,l}^+$) to represent
the left and right limits of $\bv_h \in \bV_h$ ($v_{h,l}\in V_h$) respectively. We
also introduce notations $\{\bv_h\} = \frac{1}{2}(\bv_h^++\bv_h^-)$ ($\{v_{h,l}\}
= \frac{1}{2}(v_{h,l}^++v_{h,l}^-)$) for the
averages and $[\bv_h] = \bv_h^+-\bv_h^-$ ($[v_h] =
v_h^+-v_h^-$) for the jumps. 

Let $\{x_i^r\}_{r = 1}^{k+1}$ be the $k+1$ Gauss-Lobatto quadrature points
on $I_i$ and $\{w_i^r\}_{r = 1}^{k+1}$ be the corresponding weights 
associated with the
normalized interval $[-1,1]$. We
introduce the following notations for the quadrature rule.
\[\ninti \bvphi\cdot \bpsi dx := \frac{h_i}{2}\sum_{r = 1}^{k+1} w_r
\bvphi(x_i^r)\cdot \bpsi(x_i^r) .\]
\[\ninti \bvphi\cdot \partial_x\bpsi dx := \frac{h_i}{2}\sum_{r = 1}^{k+1} w_r
\bvphi(x_i^r)\cdot\partial_x(\mI \bpsi)(x_i^r) .\]
Here $\mI$ is a component-wise interpolation operator. Namely, $\mI\bpsi: \RR
\to \RR^m$, and the $l$-th component of $\mI\bpsi$ is the $k$-th order interpolation polynomial of $\psi_l$ at Gauss-Lobatto 
points. We also define $\nintO \cdot dx = \sum_{i=1}^N \ninti \cdot dx$. 

\subsection{Semi-discrete scheme and entropy stability}
We firstly rewrite \eqref{eq-1d} into a first-order system.
  \begin{align*}
    \partial_t \brho& = \partial_x (F\bu) ,\\
    \bu& = \partial_x \bxi.
  \end{align*}
On each mesh cell $I_i$, we multiply with a test function and integrate by parts
to get
  \begin{align*}
    \inti \partial_t \brho \cdot \bvphi dx&= - \inti  F \bu \cdot \partial_x \bvphi
    dx + (F\bu \cdot \bvphi)_{i+\hf}^- - (F\bu\cdot \bvphi)_{i-\hf}^+
    ,
    \\
    \inti \bu\cdot \bpsi dx &= -\inti \bxi\cdot \partial_x \bpsi dx +
    (\bxi\cdot \bpsi)_{i+\hf}^- - (\bxi\cdot \bpsi)_{i-\hf}^+.
  \end{align*}
The numerical scheme is obtain by taking trial and test functions from the
finite element space, replacing cell interface values with numerical fluxes and
applying the quadrature rule. More precisely, we seek
$\brho_h, \bu_h \in \bV_h$ such that for all $\bvphi_h, \bpsi_h \in \bV_h$, 
\begin{subequations}\label{eq-1dscheme}
\begin{align}
    \ninti \partial_t \brho_h \cdot \bvphi_h dx&= - \ninti  F_h\bu_h \cdot \partial_x
    \bvphi_h dx + (\wFbu \cdot \bvphi_h^-)_{i+\hf} - (\wFbu \cdot \bvphi_h^+)_{i-\hf} 
    ,\label{eq-1dscheme-1}\\
    \ninti \bu_h\cdot \bpsi_h dx &= -\ninti \bxi_h\cdot \partial_x \bpsi_h dx +
    (\wbxi\cdot \bpsi_h^-)_{i+\hf} - (\wbxi \cdot
    \bpsi_h^+)_{i-\hf}.\label{eq-1dscheme-2}
\end{align}
\end{subequations}
Here $F_h = F(\brho_h)$, $\bxi_h = \bxi(\brho_h)$, $\brho_h =
(\rho_{h,1},\cdots,\rho_{h,m})^T$ and $\bu_h$, $\bvphi_h$, $\bpsi_h$ are defined
similarly. We have the following choices for $\wFbu$ and
$\wbxi$.

(1) Central and Lax-Friedrichs fluxes:
\begin{equation}\label{eq-1dflux}
  \wbxi = \{\bxi_h\},\qquad \wFbu = \{F_h\bu_h\} + \frac{\alpha}{2}[\brho_h]=\{\diag(\brho_h)\bv_h\} + \frac{\alpha}{2}[\brho_h],
\end{equation}
where
\[\alpha = \max(|\bv_{h}^+|_\infty,|\bv_{h}^-|_\infty),\qquad \bv_h =
\diag(\brho_h)^{-1}F_h\bu_h,\]
and $|\cdot|_\infty$ is the $l^\infty$ norm on $\RR^m$. 

(2) Alternating fluxes:
\begin{align}\label{eq-1dflux-alter}
  \wbxi = \bxi_h^\mp,\qquad \wFbu = (F_h\bu_h)^\pm.
\end{align}

\begin{REM}
 Due to \eqref{eq-1dflux}, the scheme \eqref{eq-1dscheme-1} is
  equivalent to
  \[\ninti \partial_t \rho_{h,l} \vphi_{h,l} dx = - \ninti
    \rho_{h,l}v_{h,l} \partial_x
    \vphi_{h,l} dx +  (\widehat{\rho v}_l \vphi^-_{h,l})_{i+\hf} -
  (\widehat{\rho v}_{l}\vphi^+_{h,l})_{i-\hf} \]
  with \[\widehat{\rho v}_l = \frac{1}{2}(\rho_{h,l}^+v_{h,l}^++
  \rho_{h,l}^-v_{h,l}^-) + \frac{\alpha}{2}(\rho_{h,l}^+-\rho_{h,l}^-),\]
   which is formally reduced to the scalar case discussed in
   \cite{sun2018discontinuous}.
\end{REM}
We define the discrete entropy
\begin{equation}\label{eq-1dent}
  E_h = \nintO e(\brho_h) dx.
\end{equation}
As is stated in Theorem \ref{thm-1dent}, the numerical scheme has a decaying entropy as that for the continuum
system. 

\begin{THM}[Entropy inequality]\label{thm-1dent}
  Let $\brho_h$ and $\bu_h$ be obtained from
  \eqref{eq-1dscheme}, \eqref{eq-1dflux} and \eqref{eq-1dflux-alter}. $E_h$ is the associated discrete
  entropy defined in \eqref{eq-1dent}.
Then 

(1) for central and Lax-Friedrichs fluxes,
\[\frac{d}{dt}E_h = - \overset{\sim}{\int_\Omega} \bu_h \cdot F_h\bu_h  dx
  -\frac{1}{2}\sum_{i=1}^N \alpha_{i+\hf}
[\bxi_h]_{i+\hf}\cdot [\brho_h]_{i+\hf}\leq 0,\]

(2) for alternating fluxes,
\[\frac{d}{dt}E_h = - \overset{\sim}{\int_\Omega} \bu_h \cdot F_h\bu_h  dx.\]
\end{THM}

\begin{proof}
A direct computation yields
$$
\frac{d}{dt}E_h = \sum_{i=1}^N \ninti \partial_t \brho_h \cdot \bxi_h dx\,.
$$
Then, with $\bvphi_h = \bxi_h$ in \eqref{eq-1dscheme-1}, we have
\begin{equation*}
  \begin{aligned}
    \ninti \partial_t \brho_h \cdot \bxi_h dx
    =& - \ninti F_h \bu_h \cdot
    \partial_x \bxi_h dx + (\wFbu \cdot \bxi_h^-)_{i+\hf} - (\wFbu\cdot
    \bxi_h^+)_{i-\hf} \\
    =& - \inti  \cI(F_h\bu_h) \cdot \partial_x \cI(\bxi_h) dx + (\wFbu \cdot
    \bxi_h^-)_{i+\hf} - (\wFbu\cdot \bxi_h^+)_{i-\hf}.
  \end{aligned}
\end{equation*}
Here we have used the fact that $\cI(F_h\bu_h)\cdot \partial_x \cI(\bxi_h)$ is a
polynomial of degree $2k-1$ and that the Gauss-Lobatto quadrature rule is exact.
Then again with integrating by parts and the exactness of the quadrature, one
can get
\begin{equation*}
  \begin{aligned}
    \ninti \partial_t \brho_h \cdot \bxi_h dx
    =& \inti  \partial_x\cI(F_h\bu_h) \cdot  \cI(\bxi_h) dx 
    - (F_h\bu_h \cdot \bxi_h)_{i+\hf}^- + (F_h\bu_h\cdot \bxi_h)_{i-\hf}^+ \\
    &+ (\wFbu \cdot \bxi_h^-)_{i+\hf} - (\wFbu\cdot \bxi_h^+)_{i-\hf} \\
    =& \ninti \bxi_h \cdot \partial_x (F_h\bu_h)  dx 
    - (F_h\bu_h \cdot \bxi_h)_{i+\hf}^- + (F_h\bu_h\cdot \bxi_h)_{i-\hf}^+ \\
    &+ (\wFbu \cdot \bxi_h^-)_{i+\hf} - (\wFbu\cdot \bxi_h^+)_{i-\hf} \\
    =& - \ninti \bu_h \cdot F_h\bu_h  dx + (\wbxi\cdot (F_h\bu_h)^-)_{i+\hf} -
    (\wbxi \cdot (F_h\bu_h)^+)_{i-\hf}\\
    &- (F_h\bu_h \cdot \bxi_h)_{i+\hf}^- + (F_h\bu_h\cdot \bxi_h)_{i-\hf}^+
    + (\wFbu \cdot \bxi_h^-)_{i+\hf} - (\wFbu\cdot \bxi_h^+)_{i-\hf}\,,
  \end{aligned}
\end{equation*}
where in the last identity we used the scheme \eqref{eq-1dscheme-2} with $\bpsi_h=F_h\bu_h$.
After summing over the index $i$ and  using the periodicity, we obtain
\begin{equation*}
  \begin{aligned}
    \frac{d}{dt}E_h &= \sum_{i=1}^N \ninti \partial_t \brho_h \cdot \bxi_h dx
    \\
  &= - \overset{\sim}{\int_\Omega} \bu_h \cdot F_h\bu_h  dx +
  \sum_{i=1}^{N}\left(
    ((F_h\bu_h)^- -\wFbu)_{i+\hf} \cdot [\bxi_h]_{i+\hf} +
[F_h\bu_h]_{i+\hf}\cdot (\bxi^+_h-\wbxi)_{i+\hf}\right).
\end{aligned}
\end{equation*}
The proof is completed by substituting numerical fluxes in \eqref{eq-1dflux} and
\eqref{eq-1dflux-alter}. Note that
\begin{align*}
\sum_{i=1}^N \alpha_{i+\hf}[\bxi_h]_{i+\hf}\cdot
[\brho_h]_{i+\hf} =&\,  
\sum_{i=1}^N \alpha_{i+\hf}(\nabla e ((\brho_h)_{i+\hf}^+)-\nabla e ((\brho_h)_{i+\hf}^-))\cdot
((\brho_h)_{i+\hf}^+ - (\brho_h)_{i+\hf}^-)
\\
=&\,  \sum_{i=1}^N \alpha_{i+\hf}[\brho_h]_{i+\hf}\cdot
D^2e(\bzeta_{i+\hf}) [\brho_h]_{i+\hf}\geq 0
\end{align*}
due to the convexity of
$e$. Here $\bzeta_{i+\hf}$ lies in the line segment 
between $(\brho_h)_{i+\hf}^-$ and $(\brho_h)_{i+\hf}^+$.
\end{proof}

\begin{REM}
  Both choices of numerical fluxes are entropy stable, while
  they have different advantages and disadvantages. 
  For central and Lax-Friedrichs fluxes defined in \eqref{eq-1dflux}, they
  preserve positive cell averages as time steps are small. Hence 
  the positivity-preserving limiter can be applied 
  for producing non-negative solutions. Details are given 
  in the next section. However, they are limited to problems satisfying
  $|\diag(\brho)^{-1} F(\brho)|\leq C$ with $|\cdot|$ being the matrix
  norm, so that $\bv_h$ can be well-defined.
  Furthermore, for odd-order polynomials, one would observe 
  a reduced order of accuracy with this particular
  choice of $\alpha$. See accuracy tests in Section \ref{sec-tests1} 
  and Section \ref{sec-tests2}. Alternating fluxes do not have such order
  reduction and restriction on $F$, while it may fail to preserve non-negative
  cell averages after one Euler forward step. 
\end{REM}

\begin{REM}
  Due to the possible degeneracy of the problem, in general it is not easy to
  extend the entropy decay property to fully discrete explicit schemes. 
 When $F$ is uniformly positive-definite
  and $e$ is strongly convex, a fully discrete entropy inequality can be
  derived. We postpone to Appendix \ref{sec-prooffully} a proof of such result, where we consider the Euler    forward time discretization with central and Lax-Friedrichs fluxes. 
\end{REM}

\subsection{Fully discrete scheme and preservation of positivity}

\subsubsection{Euler forward time stepping}

When central and Lax-Friedrichs fluxes are used, one can adopt the methodology
introduced by Zhang et al. in
\cite{zhang2010maximum,zhang2017positivity,srinivasanapositivity} to enforce
positivity of the solution.

It can be shown, when the solution achieves non-negative values at the current
time level, with fluxes defined in
\eqref{eq-1dflux} and
a sufficiently small time step, the Euler forward time discretization will produce
a solution with non-negative cell averages at the next time level. This is
referred to as the weak positivity. 
Then we can apply a scaling limiter, squashing the solution
polynomials towards 
cell averages to avoid inadmissible negative values. 
It is shown in \cite{zhang2017positivity} that the limiter preserves the high-order spatial
accuracy.\\

\begin{THM}\label{thm-1dposi}
  1. Suppose $\rho_{h,l}^n(x_i^r) \geq 0$ for all $i$ and $r$.
  Take \[\tau \leq \min_{\substack{i=1,\ldots,N}} \left(
          \left(\frac{w_1h_i}{\alpha+v_{h,l}^+}\right)_{i-\hf},
      \left(\frac{w_{k+1}h_i}{\alpha-v_{h,l}^-}\right)_{i+\hf}\right),\]
    \item then the preliminary solution $\rho_{h,l}^{n+1,\pre}$ obtained through the Euler forward time
    stepping 
    \[\ninti \frac{\brho_{h}^{n+1,\pre}-\brho_h^n}{\tau} \cdot \bvphi dx = 
      -\ninti F\bu \cdot \partial_x \bvphi dx + (\wFu\cdot \bvphi_h^-)_{i+\hf} -
    (\wFu\cdot \bvphi_h^+)_{i-\hf}.\]
    has non-negative cell averages.
    
    2. Let \[\theta_{l,i} =
      \min\left(\frac{(\barrho_{h,l})_i^{n+1,\pre}}{(\barrho_{h,l})_i^{n+1,\pre}-
      \rho^{\min}_{l,i}}, 1\right), \qquad \rho^{\min}_{l,i} =
    \min_r\left(\rho_{h,l}^{n+1,\pre}(x_i^r)\right). \]
    Define $\rho_{h,l}^{n+1}$ so that 
    \[\rho_{h,l}^{n+1}(x_i^r) = (\barrho_{h,l})_i^{n+1,\pre} +
    \theta_{l,i}(\rho_{h,l}^{n+1,\pre}(x_i^r) - (\barrho_{h,l})_i^{n+1,\pre}).\]
    Then $\rho_{h,l}^{n+1}(x_i^r)$ is non-negative. Furthermore, the interpolation
    polynomial $\rho_{h,l}^{n+1}(x)$ maintains spatial accuracy in the sense that
    \begin{equation*}
      |(\rho_{h,l})_i^{n+1}(x)-(\rho_{h,l})_i^{n+1,\pre}(x)|\leq C_k\max_{x\in
    I_i}|\rho_l(x,t_{n+1})-(\rho_{h,l})_i^{n+1,\pre}(x)|,
    \end{equation*}
    where $C_k$ is a constant depending only on the polynomial degree $k$. 

\end{THM}

\begin{proof}
  Note $\wFbu = \frac{1}{2}\left(\diag(\brho_h^+)\bv_h^+ +
  \diag(\brho_h^-)\bv_h^-\right) +
  \frac{\alpha}{2}(\brho_h^+-\brho_h^-)$. 
  Hence
  \[(\widehat{\rho v})_l := (\wFbu)_l =
    \hf\left(\rho_{h,l}^+v_{h,l}^++\rho_{h,l}^-v_{h,l}^-\right) +
  \frac{\alpha}{2}(\rho_{h,l}^+-\rho_{h,l}^-).\]
  Then
  \[\frac{(\barrho_{h,l})_{i}^{n+1}-(\barrho_{h,l})_i^n}{\tau} =
    \frac{(\wFbu)_{l,i+\hf}^n-(\wFbu)_{l,i-\hf}^n}{h} =
    \frac{(\widehat{\rho v})_{l,i+\hf}^n
  -(\widehat{\rho v})_{l,i-\hf}^n}{h},\]
We now invoke Lemma 2.1 in \cite{sun2018discontinuous} for the scalar case implying 
$(\barrho_{h,l})_i^{n+1}\geq 0$. The non-negativity of the solution
is ensured through the definition of $\theta_{l,i}$. For the accuracy result
we refer to Theorem 4 in \cite{zhang2017positivity}.
\end{proof}

\begin{REM}[Pointwise non-negativity]
  The procedure stated in Theorem \ref{thm-1dposi} only guarantees the
  non-negativity of the solution at all Gauss-Lobatto points. 
  This would be enough for most scenarios, since the scheme is 
  defined only through these points. 
  One can also ensure pointwise non-negativity of solution polynomials on the
  whole interval by
  taking $\rho^{min}_{i,l} = \inf_{x\in I_i} \rho_{h,l}^{n+1,\pre}(x)$.
\end{REM}
\begin{REM}[Time steps] As has been analyzed for the scalar case in
  Remark 2.2 in \cite{sun2018discontinuous}, we expect the time step to be $\tau \leq \mu h^2$ for some
  constant $\mu$. In practice, we assume a diffusion number $\mu$. If a negative
  cell average emerges, we halve the time step and redo the computation. Theorem
  \ref{thm-1dposi} guarantees that the halving procedure will end after finite
  loops. 
\end{REM}

\begin{REM}[The scaling parameter]
  For robustness of the algorithm, especially for dealing with log-type
  entropy functionals, we introduce a parameter 
  $\veps_{l,i} = \min\left(10^{-13}, (\barrho_{h,l})_i^{n+1,\pre}\right)$ and use
  $\theta^\veps_{l,i} = 
  \min\left(\frac{(\barrho_{h,l})_i^{n+1,\pre}-\veps_{l,i}}{(\barrho_{h,l})_i^{n+1,\pre}-\rho^{\min}_{l,i}},1\right)$ 
  instead of $\theta$ in our computations. In other words,
  if $(\barrho_{h,l})_i^{n+1,\pre}>10^{-13}$, we scale the polynomial to require
  it takes values not  smaller than $10^{-13}$ at Gauss-Lobatto points; 
  otherwise, we set the solution to be the constant $(\barrho_{h,l})_i^{n+1,\pre}$ in the
  interval. Note as long as $\barrho_{h,l}(x)\geq 10^{-13}$, the accuracy could still be
  maintained. 
\end{REM}
\begin{REM}[Other bounds]
  In general, it would be difficult to preserve other bounds besides
  positivity though this procedure. Similar difficulty has also been
  encountered in \cite{srinivasanapositivity}. 
\end{REM}
\subsubsection{High-order time discretization}\label{sec-highorder}

We adopt SSP-RK methods for high-order time discretizations.
Since the time step will be chosen as $\tau = \mu h^2$,
the first-order Euler forward time stepping would be sufficiently accurate for $k = 1$ to
achieve the second-order accuracy. For $k = 2$ and $k=3$, we apply 
the second-order SSP-RK method 
\begin{align*}
  \brho_h^{(1)} &= \brho_h^n + \tau \bF(\brho_h^n),\\
  \brho_h^{n+1} &= \frac{1}{2}\brho_h^n + \frac{1}{2}\left(\brho_h^{(1)} + \tau
  \bF(\brho_h^{(1)})\right).
\end{align*}
For $k = 4$ and $k = 5$, a third-order time discretization method should be
used
\begin{align*}
  \brho_h^{(1)} &= \brho_h^n + \tau \bF(\brho_h^n),\\
  \brho_h^{(2)} &= \frac{3}{4}\brho_h^n + \frac{1}{4}\left(\brho_h^{(1)} + \tau
  \bF(\brho_h^{(1)})\right),\\
  \brho_h^{n+1} &= \frac{1}{3}\brho_h^n + \frac{2}{3}\left(\brho_h^{(2)} + \tau
  \bF(\brho_h^{(2)})\right).
\end{align*}

As one can see, SSP-RK methods can be rewritten as convex combinations of Euler forward steps.
Hence if the time step is chosen to be sufficiently small, and the scaling
limiter is applied at each Euler forward stage, then the solution will remain non-negative after one full time step. 


\section{Numerical method: two-dimensional case}\label{sec-2dnum}
\setcounter{equation}{0}
\setcounter{figure}{0}
\setcounter{table}{0}
\subsection{Notations}
In this section, we generalize the previous ideas to two-dimensional systems of the form
\begin{equation*}
  \partial_t \brho = \partial_x (F \partial_x \bxi) + \partial_y (F
  \partial_y \bxi).
\end{equation*}
The problem is set on a rectangular domain $\Omega = I\times J$ with the compact support
or periodic boundary condition. We consider a regular Cartesian mesh on
$\Omega$, with $I =
\cup_{i=1}^{N^x} I_i$, $I_i = (x_{i-\hf}, x_{i+\hf})$, $J = \cup_{j = 1}^{N^y}
J_j$ and $J_j = (y_{j-\hf},y_{j+\hf})$. Let $h_i^x = x_{i+\hf}-x_{i-\hf}$ and
$h_j^y = y_{j+\hf}-y_{j-\hf}$. Then the solution is sought in the
following finite element space.
\[\bV_h = \prod_{l=1}^N V_h,\qquad V_h = \{v_h: v_h|_{I_i\times J_j} \in Q^k(I_i\times
J_j)\}.\]
$Q^k(I_i\times J_j)$ is the tensor product space of $P^k(I_i)$ and
$P^k(J_j)$. 
For $\bv_h(x,y) \in \bV_h$, 
\[\{\bv_h\}_{i+\hf}(y) =
\frac{1}{2}\left(\bv_h(x^+_{i+\hf},y)+\bv_h(x^-_{i+\hf},y)\right),\]
\[\{\bv_h\}_{j+\hf}(x) =
\frac{1}{2}\left(\bv_h(x,y^+_{j+\hf})+\bv_h(x,y^-_{j+\hf})\right),\] 
\[[\bv_h]_{i+\hf}(y) = \bv_h(x^+_{i+\hf},y)-\bv_h(x^-_{i+\hf},y),\]
\[[\bv_h]_{j+\hf}(x) = \bv_h(x,y^+_{j+\hf})-\bv_h(x,y^-_{j+\hf}).\] 
Same notations will be used in the scalar case $v_{h,l} \in V_h$. 
Finally, for the quadrature rule, we  denote by $\nintJ \cdot dy = \sum_{j=1}^{N^y}
\nintj \cdot dy$ and $\nintI \cdot dx = \sum_{i=1}^{N^x} \ninti \cdot dx$. 

\subsection{Semi-discrete scheme and entropy inequality}
The semi-discrete scheme is given as follows. Find $\brho_h$, $\bu^x_h$ and
$\bu^y_h$ in $\bV_h$, such that for all $\bvphi_h$, $\bpsi^x_h$ and $\bpsi^y_h$
in $\bV_h$, we have
  \begin{align*}
    \nintj\ninti \partial_t \brho_h \cdot \bvphi_h dx dy= &\, - \nintj \ninti \left( F_h
    \bu_h^x
    \cdot \partial_x \bvphi_h + F_h \bu_h^y \cdot \partial_y \bvphi_h \right)
  dx dy\\
  &+ \nintj \left( \wFbux_{i+\hf}(y) \cdot \bvphi_h(x_{i+\hf}^-,y) -
  \wFbux_{i-\hf}(y)\cdot
  \bvphi_h(x_{i-\hf}^+,y) \right)dy\\
  &+ \ninti \left(\wFbuy_{j+\hf}(x) \cdot\bvphi_h(x,y_{j+\hf}^-) -
  \wFbuy_{j-\hf}(x) \cdot \bvphi_h(x,y_{j-\hf}^+) \right)dx ,\\
   \nintj \ninti \bu^x_h\cdot \bpsi_h^x + \bu^y_h \cdot \bpsi_h^y dx dy = &\,
   -\nintj \ninti \left(\bxi_h\cdot \left(\partial_x \bpsi_h^x + \partial_y
   \bpsi_h^y\right)  \right)dxdy \\
   &+\nintj \left(\wbxi_{i+\hf}(y) \cdot \bpsi_h^x(x_{i+\hf}^-,y) -
   \wbxi_{i-\hf}(y)\cdot
 \bpsi_h^x(x_{i-\hf}^+,y) \right)dy \\
 &+\ninti \left(\wbxi_{j+\hf}(x) \cdot \bpsi_h^y(x,y_{j+\hf}^-) -
 \wbxi_{j-\hf}(x)\cdot
 \bpsi_h^y(x,y_{j-\hf}^+) \right)dx.
  \end{align*}
As that in the one-dimensional case, two choices of numerical fluxes can be used.

(1) Central and Lax-Friedrichs flux
\begin{align*}
  \wbxi_{i+\hf}(y) &= \{\bxi_h\}_{i+\hf}(y),\qquad 
  \wbxi_{j+\hf}(x) = \{\bxi_h\}_{j+\hf}(x),\\
  \wFbux_{i+\hf}(y) &= \left(\{F_h\bu_h^x\}_{i+\hf} +
  \frac{\alpha^x_{i+\hf}}{2}[\brho_h]_{i+\hf}\right)(y),\\
  \alpha^x_{i+\hf}(y) &=
  \max(|\bv_h^x(x_{i+\hf}^+,y)|_\infty,|\bv_h^x(x_{i+\hf}^-,y)|_\infty),\qquad 
  \bv_h^x = \diag(\brho_h)^{-1} F_h\bu_h^x,\\
  \wFbuy_{j+\hf}(x)&= \left(\{F_h\bu_h^y\}_{j+\hf}+
  \frac{\alpha^y_{j+\hf}}{2}[\brho_h]_{j+\hf}\right)(x),\\
  \alpha^y_{j+\hf}(x) &=
  \max(|\bv_h^y(x,y_{j+\hf}^+)|_\infty,|\bv_h^y(x,y_{j+\hf}^-)|_\infty),\qquad 
  \bv_h^y = \diag(\brho_h)^{-1} F_h\bu_h^y.
\end{align*}
(2) Alternating fluxes
\begin{align*}
  \wbxi_{i+\hf}(y) &= \bxi_h(x_{i+\hf}^\mp,y), \qquad 
  \wFbux_{i+\hf}(y) = (F_h\bu_h^x)(x_{i+\hf}^\pm,y),\\
  \wbxi_{j+\hf}(x) &= \bxi_h(x,y_{j+\hf}^\mp),\qquad
  \wFbuy_{j+\hf}(x) = (F_h\bu_h^y)(x,y_{j+\hf}^\pm).
\end{align*}
The numerical scheme mimics a similar entropy decay behavior as the continuum
equation. We define the discrete entropy 
\begin{equation}
  E_h = \nintJ\nintI e(\brho_h) dx dy. \label{eq-de}
\end{equation}
We state the following entropy decay property for the semi-discrete scheme. 

\begin{THM}\label{thm-2dent}
  Let $\brho_h$, $\bu_h^x$ and $\bu_h^y$ be obtained from the semi-discrete
  scheme for two-dimensional problems. $E_h$ defined in \eqref{eq-de} is the discrete
  entropy. 

  (1)  Suppose central and Lax-Friedrichs fluxes are used, then
\begin{equation*}
  \begin{aligned}
  \frac{d}{dt}E_h =& - \nintJ\nintI \left(\bu_h^x \cdot F_h\bu_h^x + \bu_h^y\cdot F_h
  \bu_h^y\right)
  dxdy\\
  &- \frac{1}{2}\sum_{i=1}^{N^x}
\nintJ\alpha_{i+\hf}^x[\bxi_h]_{i+\hf}(y)\cdot[\brho_h]_{i+\hf}(y)
dy \\
&- \frac{1}{2}\sum_{j=1}^{N^y}
\nintI\alpha_{j+\hf}^y
[\bxi_h]_{j+\hf}(x)\cdot 
[\brho_h]_{j+\hf}(x)
dx\leq 0.
  \end{aligned}
\end{equation*}

  (2) Suppose alternating fluxes are used, then
\begin{equation*}
  \frac{d}{dt}E_h = - \nintJ\nintI \left(\bu_h^x \cdot F_h\bu_h^x + \bu_h^y\cdot F_h
  \bu_h^y\right)
  dxdy\leq 0.
\end{equation*}
\end{THM}

\begin{proof}
Following the blueprint of the one dimensional case, we get by a direct computation that
 $$
 \frac{d}{dt}E_h = \sum_{j=1}^{N^y} \sum_{i = 1}^{N^x}  \nintj\ninti \partial_t
  \brho_h \cdot \bxi_h dxdy\,.
 $$
 Then using the scheme, we deduce
\begin{equation*}
  \begin{aligned}
    \nintj \ninti \partial_t \brho_h \cdot \bxi_h dx dy =& - \nintj \ninti \left( F_h \bu_h^x \cdot
    \partial_x \bxi_h + F_h \bu_h^y \cdot \partial_y \bxi_h \right) dxdy\\
  &+ \nintj \left( \wFbux_{i+\hf} \cdot \bxi_h(x_{i+\hf}^-,y) - \wFbux_{i-\hf}\cdot
  \bxi_h(x_{i-\hf}^+,y) \right) dy\\
  &+ \ninti \left( \wFbuy_{j+\hf} \cdot \bxi_h(x,y_{j+\hf}^-) -
  \wFbuy_{j-\hf} \cdot \bxi_h(x,y_{j-\hf}^+) \right) dx \,.\\
  \end{aligned}
\end{equation*}
For each fixed $y$, each component of $\cI(F_h \bu^x)\partial_x \cI(\bxi_h)$ is
a polynomial of degree $2k-1$ with respect to $x$. Analogously, $\cI(F_h \bu^y)\partial_y
\cI(\bxi_h)$ is a $(2k-1)$-th order polynomial with respect to $y$ for each fixed
$x$. Hence the Gauss-Lobatto quadrature with $k+1$ node
is exact. We then replace the quadrature rule with the exact integral,
and integrate by parts to get
\begin{equation*}
  \begin{aligned}
    \frac{d}{dt}E 
  = &\sum_{i=1}^{N^x}\sum_{j=1}^{N^y} \left[ - \nintj \inti  \cI(F_h\bu_h^x)
    \cdot \partial_x \cI(\bxi_h) dx dy -
  \ninti \intj \cI(F_h \bu_h^y) \cdot \partial_y \cI(\bxi_h)
dydx \right.\\
  &\qquad\qquad+ \nintj \left( \wFbux_{i+\hf} \cdot \bxi_h(x_{i+\hf}^-,y) - \wFbux_{i-\hf}\cdot
  \bxi_h(x_{i-\hf}^+,y)\right)  dy\\
  &\qquad\qquad\left.+ \ninti \left( \wFbuy_{j+\hf} \cdot \bxi_h(x,y_{j+\hf}^-) -
  \wFbuy_{j-\hf} \cdot \bxi_h(x,y_{j-\hf}^+) \right) dx\right]\\
    =&\sum_{i=1}^{N^x}\sum_{j=1}^{N^y}\left[ \nintj \inti
        \partial_x\cI(F_h\bu_h^x) \cdot
    \cI(\bxi_h) dx dy \right. +\ninti \intj \partial_y \cI(F_h \bu_h^y) \cdot \cI(\bxi_h)
dydx\\
&\qquad\qquad- \nintj \left( (F_h\bu_h^x\cdot \bxi_h)(x_{i+\hf}^-,y) - (F_h\bu_h^x\cdot \bxi_h)(x_{i-\hf}^+,y) \right) dy\\
&\qquad\qquad\left.- \ninti \left( (F_h\bu_h^y \cdot \bxi_h)(x,y_{j+\hf}^-) -
  (F_h\bu_h^y \cdot \bxi_h)(x,y_{j-\hf}^+) \right) dx\right.\\
  &\qquad\qquad+ \nintj \left( \wFbux_{i+\hf} \cdot \bxi_h(x_{i+\hf}^-,y) - \wFbux_{i-\hf}\cdot
  \bxi_h(x_{i-\hf}^+,y) \right) dy\\
  &\qquad\qquad\left.+ \ninti \left( \wFbuy_{j+\hf} \cdot \bxi_h(x,y_{j+\hf}^-) -
  \wFbuy_{j-\hf} \cdot \bxi_h(x,y_{j-\hf}^+) \right) dx\right].
  \end{aligned}
\end{equation*}
Again by changing back to the quadrature rule and applying the scheme, we obtain
\begin{equation*}
  \begin{aligned}
    \frac{d}{dt}E 
    =&\sum_{i=1}^{N^x}\sum_{j=1}^{N^y}\left[ 
  - \nintj\ninti  \left( \bu_h^x\cdot  F_h\bu_h^x + \bu_h^y \cdot F_h\bu_h^y \right) dxdy \right.\\
&\qquad\qquad+ \nintj \left( (F_h\bu_h^x)(x_{i+\hf}^-,y)\cdot\wbxi_{i+\hf} - (F_h\bu_h^x)(x_{i-\hf}^+,y)
\cdot \wbxi_{i-\hf} \right) dy\\
&\qquad\qquad\left.+ \ninti \left( (F_h\bu_h^y)(x,y_{j+\hf}^-)  \cdot \wbxi_{j+\hf}-
(F_h\bu_h^y)(x,y_{j-\hf}^+)\cdot \wbxi_{j-\hf} \right) dx\right.\\
&\qquad\qquad- \nintj \left( (F_h\bu_h^x\cdot \bxi_h)(x_{i+\hf}^-,y) - (F_h\bu_h^x\cdot \bxi_h)(x_{i-\hf}^+,y) \right) dy\\
&\qquad\qquad\left.- \ninti \left( (F_h\bu_h^y \cdot \bxi_h)(x,y_{j+\hf}^-) -
  (F_h\bu_h^y \cdot \bxi_h)(x,y_{j-\hf}^+) \right) dx\right.\\
  &\qquad\qquad+ \nintj \left( \wFbux_{i+\hf} \cdot \bxi_h(x_{i+\hf}^-,y) - \wFbux_{i-\hf}\cdot
  \bxi_h(x_{i-\hf}^+,y) \right) dy\\
  &\qquad\qquad\left.+ \ninti \left( \wFbuy_{j+\hf} \cdot \bxi_h(x,y_{j+\hf}^-) -
  \wFbuy_{j-\hf} \cdot \bxi_h(x,y_{j-\hf}^+) \right) dx\right].
  \end{aligned}
\end{equation*}
Since we have assumed periodicity, all cell interface terms cancel out with
alternating fluxes. 
\begin{equation*}
  \frac{d}{dt}E_h = - \nintJ\nintI \left(\bu_h^x \cdot F_h\bu_h^x + \bu_h^y\cdot F_h
  \bu_h^y\right)
  dxdy.
\end{equation*}
For central and Lax-Friedrichs fluxes, there remain additional penalty terms.
\begin{equation*}
  \begin{aligned}
  \frac{d}{dt}E_h =& - \nintJ\nintI \left(\bu_h^x \cdot F_h\bu_h^x + \bu_h^y\cdot F_h
  \bu_h^y\right)
  dxdy\\
  &- \sum_{i=1}^{N^x}
  \nintJ\frac{\alpha_{i+\hf}^x(y)}{2}\left(\bxi_h(x^+_{i+\hf},y)-\bxi_h(x^-_{i+\hf},y)\right)\cdot
\left(\brho_h(x^+_{i+\hf},y)-\brho_h(x^-_{i+\hf},y)\right)dy \\
&- \sum_{j=1}^{N^y}
\nintI\frac{\alpha_{j+\hf}^y(x)}{2}\left(\bxi_h(x,y^+_{j+\hf})-\bxi_h(x,y^-_{i+\hf})\right)\cdot
\left(\brho_h(x,y^+_{j+\hf})-\brho_h(x,y^-_{j+\hf})\right)dx \leq 0\,,
  \end{aligned}
\end{equation*}
since $F_h$ is positive-semidefinite and $e$ is convex as in the one dimensional case.
\end{proof}

\subsection{Fully discrete scheme and preservation of positivity}

One can adopt similar positivity-preserving techniques as in the one
dimensional case, when central and Lax-Friedrichs fluxes are used. 
Again, the key step is to ensure the
non-negativity of the first-order Euler forward time discretization. The
positivity will be automatically preserved by the SSP-RK time
discretization. For the first-order scenario, we could prove the following
theorem. 

  \begin{THM}\label{thm-2dposi}
    1. Suppose $\rho_{h,l}(x_i^r,y_j^s) \geq 0$ for all $i$, $j$, $r$ and $s$. Take
\begin{align*}
  \tau \leq
  \frac{1}{2}\min_{r,s,i,j}&\left(\left(\frac{w_1h_i^x}{\alpha^x+v_{h,l}^x}\right)(x_{i-\hf}^+,y_j^s), 
  \left(\frac{w_{k+1} h_i^x}{\alpha^x-v_{h,l}^x}\right)(x_{i+\hf}^-,y_j^s),\right.
  \\
  &\left. \left(\frac{w_1h_j^y}{\alpha^y+v_{h,l}^y}\right)(x_i^r,y_{j-\hf}^+),
  \left(\frac{w_{k+1}h_j^y}{\alpha^y-v_{h,l}^y}\right)(x_i^r,y_{j+\hf}^-)\right).
\end{align*}
Then the preliminary solution $(\rho_{h,l})_{i,j}^{n+1,\pre}$ obtained through the
Euler forward time discretization has non-negative cell averages. 

2. Let 
\[\theta_{l,i,j} =
  \min\left(\frac{(\barrho_{h,l})_{i,j}^{n+1,\pre}}{(\barrho_{h,l})_{i,j}^{n+1,\pre}
- m_{l,i,j}}, 1\right), \qquad  \rho^{\min}_{l,i,j} =
\min_{r,s}\left(\rho_{h,l}^{n+1,\pre}(x_{i}^r,y_j^s)\right).\]
Define $\rho_{h,l}^{n+1}$ so that
\[\rho_{h,l}^{n+1}(x_i^r, y_j^s) = (\barrho_{h,l})_{i,j}^{n+1,\pre} +
  \theta_{l,i,j}(\rho_{h,l}^{n+1,\pre}(x_i^r,y_j^s) -
(\barrho_{h,l})_{i,j}^{n+1,\pre} ).\]
Then $\rho_{h,l}^{n+1}(x_i^r,y_j^s)$ is non-negative. Furthermore, $\rho_{h,l}^{n+1}$ would
maintain the spatial accuracy achieved by $\rho_{h,l}^{n+1,\pre}$.
\end{THM}

\begin{proof}
Using the scheme we obtain
\begin{equation*}
\begin{aligned}
  \frac{(\barrho_{h,l})_{i,j}^{n+1,\pre} -(\barrho_{h,l})_{i,j}^n}{\tau} 
  =& 
    \frac{1}{h_i^xh_j^y}\nintj (\wFbux)_{l,i+\hf}-(\wFbux)_{l,i-\hf} dy\\
    &+
  \frac{1}{h_i^xh_j^y}\ninti (\wFbuy)_{l,j+\hf}-(\wFbuy)_{l,j-\hf} dx.
\end{aligned}
\end{equation*}
  Note that the numerical fluxes can be rewritten as
  \[\wFbux =
    \frac{1}{2}\left((\diag(\brho_h)\bv_h^x)^++(\diag(\brho_h^-)\bv_h^x)^-\right) +
  \frac{\alpha^x}{2}\diag(\brho_h^+-\brho_h^-),\]
  \[\wFbuy =
    \frac{1}{2}\left( (\diag(\brho_h^+)\bv_h^y)^++(\diag(\brho_h^-)\bv_h^y)^-\right) +
  \frac{\alpha^y}{2}\diag(\brho_h^+-\brho_h^-).\]
  Let 
  \[(\widehat{\rho v^x})_l := (\wFbux)_l =
    \frac{1}{2}\left( (\rho_{h,l}v_{h,l}^x)^+ + (\rho_{h,l}v_{h,l}^x )^-\right)
  + \frac{1}{2}\alpha^x(\rho_{h,l}^+ - \rho_{h,l}^-),\]
  and
  \[(\widehat{\rho v^y})_l := (\wFbuy)_l =
    \frac{1}{2}\left( (\rho_{h,l}v_{h,l}^y)^+ + (\rho_{h,l}v_{h,l}^y )^-\right)
  + \frac{1}{2}\alpha^y(\rho_{h,l}^+ - \rho_{h,l}^-).\]
  Then we have
  \begin{equation*}
    \begin{aligned}
    (\barrho_{h,l})_{i,j}^{n+1,\pre}
    =& (\barrho_{h,l})_{i,j}^n + 
    \frac{\tau}{h_i^xh_j^y}\nintj (\widehat{\rho v^x})_{l,i+\hf}-(\widehat{\rho
    v^x})_{l,i-\hf} dy \\
    &+ \frac{\tau}{h_i^xh_j^y}\ninti (\widehat{\rho v^y})_{l,j+\hf}-(\widehat{\rho
  v^y})_{l,j-\hf} dx,
\end{aligned}
\end{equation*}
  which formally reduces to the scalar case. One can follow the proof of Theorem
  3.1 in \cite{sun2018discontinuous} to show that $\barrho_{h,l}^{n+1,\pre}\geq 0$, if the prescribed
  time step restriction is satisfied. The non-negativity of
  $\rho_{h,l}^{n+1}(x_i^r,y_j^s)$ can be justified using the definition of
  $\theta_{l,i,j}$.
\end{proof}

\begin{REM}
  As before, pointwise non-negative solutions can be obtained by taking
  $\rho^{\min}_{l,i,j} = \inf_{(x,y)\in I_i\times J_j}\rho_{h,l}^{n+1,\pre}(x,y)$. The time step $\tau = \mu h^2$ will be used for time marching. As
  negative cell averages emerge, we halve the time step. We will also use
  $\theta_{l,i,j}^\veps = \min\left(\frac{(\barrho_{h,l})_{i,j}^{n+1,\pre}-\veps_{l,i,j}}{(\barrho_{h,l})_{i,j}^{n+1,\pre}-\rho^{\min}_{l,i,j}},
  1\right)$ instead of $\theta_{l,i,j}$ when applying the scaling limiter.
  Here $\veps_{l,i,j} = \min\left(10^{-13},
  (\barrho_{h,l})_{i,j}^{n+1,\pre}\right)$.
\end{REM}

\section{One-dimensional numerical tests}\label{sec-tests1}
\setcounter{equation}{0}
\setcounter{figure}{0}
\setcounter{table}{0}
In this section, we apply the entropy stable DG method for solving one-dimensional 
problems. In Example \ref{examp:heat} and Example \ref{examp:skt}, the
accuracy of the scheme is examined. The error is measured with the discrete norm
associated with the Gauss-Lobatto quadrature rule. 
We do not invoke the
positivity-preserving limiter in either tests, to exclude order degeneracy due to the
temporal order reduction. See \cite{zhang2010maximum,srinivasanapositivity} for relevant numerical experiments.
Then we consider systems from 
tumor encapsulation (Example \ref{examp:tumor})
and 
surfactant spreading (Example \ref{examp:surf})
. In these
tests, leading fronts are formed and the issue of positivity arises. 

\begin{Examp}[Heat equations]\label{examp:heat}
  Let us first consider the initial value problem with decoupled heat equations,
\begin{equation*}
  \begin{aligned}
    \partial_t \rho_l &= \partial_{xx} \rho_l,\qquad  l = 1,2,\\   
  \end{aligned}
\end{equation*}
on $[-1,1]$ with the periodic boundary condition.
$\rho_1(x,0) = \sin\left(\pi x \right) + 2$ and $\rho_2(x,0) = \cos\left(\pi x
\right) + 2$ are taken as our initial condition. 
The system is associated with the entropy $E = \int_\Omega \rho_1(\log \rho_1-1)
+ \rho_2(\log \rho_2 - 1) dx$. Hence $\bxi = (\log \rho_1, \log \rho_2)^T$ and $F
=\diag (\brho)$. 
We compute up to $t = 0.002$, with time step set as $\tau = 0.001h^2$.
Positivity-preserving limiter is not activated. We use the Lax-Friedrichs flux and
alternating fluxes respectively for computation. As one can see from Table
\ref{tab:heat_i} and Table \ref{tab:heat_ii}, 
optimal order of convergence is achieved with the alternating fluxes. 
The order of accuracy seems to be degenerated when we use central and Lax-Friedrichs
fluxes with odd-order 
polynomials. This may related with the fact that the problem is smooth hence
the jump term $\frac{\alpha}{2}[\brho_h]$ is too small, hence the
Lax-Friedrichs flux $\wFbu$ is close to the central flux when the mesh is not
well refined, which would cause order reduction. 
In Table \ref{tab:heat_iii}, we document the error with different Lax-Friedrichs
constants, with $\alpha$ replaced by $\tilde{\alpha} = 0, 2\alpha,10\alpha$.
As one can see, the optimal order is retrieved as $\tilde{\alpha}$
becomes large. We also remark, choosing a larger Lax-Friedrichs constant does not
affect the compatibility with positivity-preserving procedure, while it does
make the time step more restrictive. 

\begin{table}[h!] 
  \centering 
  \begin{tabular}{c|c|c|c|c|c|c|c} 
    \hline   
    $k$&$N$&$L^1$ error& order&$L^2$ error& order&$L^\infty$ error& order\\ 
    \hline 
   1 &80 &8.802E-04&     -&           4.958E-04&     -&           4.226E-04&     -\\
     &160&3.948E-04&     1.16&           2.224E-04&     1.16&           1.899E-04&     1.15\\
     &320&1.609E-04&     1.30&           9.077E-05&     1.29&           7.821E-05&     1.28\\
     &640&5.647E-05&     1.51&           3.201E-05&     1.50&
     2.806E-05&     1.48\\
    \hline
    2&80 &7.760E-06&     -&           8.881E-06&     -&           1.456E-05&     -\\
    &160&9.600E-07&     3.02&           1.113E-06&     3.00&
    1.825E-06&     3.00\\
     &320&1.194E-07&     3.01&           1.395E-07&     3.00&
     2.285E-07&     3.00\\
     &640&1.489E-08&     3.00&           1.745E-08&     3.00&
     2.859E-08&     3.00\\
    \hline
    3&80 &6.027E-07&     -&           4.185E-07&     -&           8.166E-07&     -\\
     &160&5.914E-08&     3.35&           4.133E-08&     3.34&
     8.365E-08&     3.29\\
     &320&5.003E-09&     3.56&           3.551E-09&     3.54&           7.567E-09&     3.47\\
     &640&3.702E-10&     3.76&           2.685E-10&     3.73&           6.054E-10&     3.64\\
    \hline
    4&80 &6.210E-10&     -&           6.542E-10&     -&           2.194E-09&     -\\
     &160&1.887E-11&     5.04&           2.030E-11&     5.01&           6.809E-11&     5.01\\
     &320&5.823E-13&     5.02&           6.338E-13&     5.00&           2.128E-12&     5.00\\
     &640&1.808E-14&     5.01&           1.981E-14&     5.00&
     6.653E-14&    5.00\\
    \hline 
  \end{tabular} 
  \caption{Accuracy test of heat equations in Example \ref{examp:heat}, with
  central flux for
  $\wbxi$ and Lax-Friedrichs flux for $\wFbu$.} 
    \label{tab:heat_i}
\end{table} 
\begin{table}[h!] 
  \centering 
  \begin{tabular}{c|c|c|c|c|c|c|c} 
    \hline   
    $k$&$N$&$L^1$ error& order&$L^2$ error& order&$L^\infty$ error& order\\ 
    \hline 
    1&80 & 4.027E-03&     -    &  2.214E-03&     -    &   1.591E-03&     -\\
     &160& 1.006E-03&     2.00&  5.530E-04&     2.00&   4.004E-04&     1.99\\
     &320& 2.515E-04&     2.00&  1.382E-04&     2.00&   1.003E-04&     2.00\\
     &640& 6.288E-05&     2.00&  3.455E-05&     2.00&   2.508E-05&     2.00\\
    \hline 
2    &80 & 1.541E-05&     -    &  1.418E-05&     -    &   3.049E-05&     -\\
     &160& 1.912E-06&     3.01&  1.769E-06&     3.00&   3.736E-06&     3.03\\
     &320& 2.383E-07&     3.00&  2.211E-07&     3.00&   4.624E-07&     3.01\\
     &640& 2.975E-08&     3.00&  2.763E-08&     3.00&   5.752E-08&     3.01\\
    \hline 
3     &80& 1.219E-07&     -    &  1.116E-07&     -   &   4.006E-07&     -\\
&160& 7.609E-09&     4.00&  6.966E-09&     4.00&   2.513E-08&     4.00\\
&320& 4.754E-10&     4.00&  4.353E-10&     4.00&   1.572E-09&     4.00\\
&640& 2.971E-11&     4.00&  2.720E-11&     4.00&   9.828E-11&     4.00\\
    \hline 
4     &80 & 1.09E-09&    -   &  1.08E-09&      -    &   4.41E-09&     -    \\
     &160& 3.401E-11&     5.00&  3.365E-11&     5.00&   1.374E-10&     5.00\\
     &320& 1.062E-12&     5.00&  1.052E-12&     5.00&   4.278E-12&     5.01\\
     &640& 3.319E-14&     5.00&  3.286E-14&     5.00&   1.334E-13&     5.00\\
    \hline 
  \end{tabular} 
  \caption{Accuracy test of heat equations in Example \ref{examp:heat}, with
    alternating fluxes $\wbxi = \bxi_h^-$ and $\wFbu =
  (F_h \bu_{h})^+$.}
    \label{tab:heat_ii}
\end{table} 

\begin{table}[h!] 
  \centering 
  \begin{tabular}{c|c|c|c|c|c|c|c} 
    \hline   
    $\tilde{\alpha}$&$N$&$L^1$ error& order&$L^2$ error& order&$L^\infty$ error& order\\ 
    \hline 
    0&80 &8.029E-07& -    &  5.549E-07&     -    &   1.031E-06&     -    \\ 
    &160&1.007E-07& 3.00&  6.949E-08&     3.00&   1.294E-07&     2.99\\ 
     &320&1.259E-08& 3.00&  8.689E-09&     3.00&   1.618E-08&     3.00\\ 
     &640&1.575E-09& 3.00&  1.086E-09&     3.00&   2.022E-09&     3.00\\ 
    \hline 
$2\alpha$&80 &4.721E-07&     -&           3.300E-07&     -&           6.683E-07&     -\\
        &160 &4.003E-08&     3.56&           2.841E-08&     3.54&           6.043E-08&     3.47\\
        &320 &2.952E-09&     3.76&           2.148E-09&     3.73&           4.843E-09&     3.64\\
        &640 &2.028E-10&     3.86&           1.502E-10&     3.84&
        3.560E-10&     3.77\\
    \hline 
$10\alpha$&80 &1.572E-07&     -&           1.150E-07&     -&           2.605E-07&     -\\    
         &160   &1.053E-08&     3.90&           7.906E-09&     3.86&           1.892E-08&     3.78\\
         &320   &6.830E-10&     3.95&           5.221E-10&     3.92&           1.300E-09&     3.86\\
         &640   &4.356E-11&     3.97&           3.366E-11&     3.96&
         8.619E-11&     3.92\\
    \hline 
  \end{tabular} 
  \caption{Accuracy test of heat equations in Example \ref{examp:heat}, with central flux for
    $\wbxi$ and Lax-Friedrichs flux for $\wFbu = \{F_h\bu_h\}
    + \frac{\tilde{\alpha}}{2}[\brho_h]$. Here $\tilde{\alpha} =
0,2\alpha,10\alpha$.}\label{tab:heat_iii}
\end{table} 
\end{Examp}
\begin{Examp}[SKT population model]\label{examp:skt}
  We use the population model of Shigesada, Kawashima and
  Teramoto \cite{shigesada1979spatial} for the second accuracy test. All the cross-diffusion and
  self-diffusion coefficients are taken as $1$. The system is written as follows. 
  \begin{subequations}\label{eq-examp-skt}
     \begin{empheq}[left=\empheqlbrace]{align}
    \partial_t \rho_1 =\partial_x\left( (2\rho_1 + \rho_2)\partial_x \rho_1 +
    \rho_1\partial_x\rho_2\right),\\
    \partial_t \rho_2 = \partial_x\left(\rho_2 \partial_x \rho_1 +
    (\rho_1+2\rho_2)\partial_x\rho_2\right).
  \end{empheq}
  \end{subequations}
\eqref{eq-examp-skt} is governed by the entropy $E = \int_\Omega
\rho_1(\log \rho_1-1) + \rho_2(\log \rho_2 -1) dx$. Again $\bxi = (\log \rho_1,
\log \rho_2)^T$. $F =\diag(\brho)\left(
\begin{matrix}2\rho_1+\rho_2&\rho_2\\\rho_1&2\rho_2+\rho_1\end{matrix}\right)$
and $\bz \cdot F\bz = 2\rho_1^2z_1^2 + 2\rho_2^2z_2^2 +
\rho_1\rho_2(z_1+z_2)^2\geq 0$. 
The computational domain is taken as $[-\pi,\pi]$. Let $\rho_1(x,0) = e^{\hf\sin
x}$ and $\rho_2(x,0) = e^{\hf\cos(2x)}$.
We assume periodic boundary condition and compute to $t = 0.2$ with $\tau =
0.0002h^2$. The numerical solution at the next mesh level is set as a reference
to evaluate the error. The error with central and Lax-Friedrichs fluxes is documented in
Table \ref{tab:skt_i}, and that with alternating fluxes is in Table
\ref{tab:skt_ii}. The exhibited order of accuracy is similar to that in Example
\ref{examp:heat}. We again compute the error for $k = 3$ with different
constants
in the Lax-Friedrichs flux. As one can see, the order of accuracy gets close to
$4$ as $\tilde{\alpha}$ increases. 
\begin{table}[h!] 
  \centering 
  \begin{tabular}{c|c|c|c|c|c|c|c} 
    \hline   
    $k$&$N$&$L^1$ error& order&$L^2$ error& order&$L^\infty$ error& order\\ 
    \hline 
    1&20 &2.852E-02&     -&           1.009E-02&     -&           8.675E-03&     -\\
     &40 &9.370E-03&     1.61&           3.461E-03&     1.54&           2.918E-03&     1.57\\
     &80 &3.031E-03&     1.63&           1.149E-03&     1.59&           9.292E-04&     1.65\\
     &160&9.527E-04&     1.67&           3.609E-04&     1.67&
     2.765E-04&     1.75\\
    \hline
    2&20 &1.093E-03&     -&           4.283E-04&     -&           4.440E-04&     -\\
     &40 &1.022E-04&     3.42&           4.329E-05&     3.31&           4.291E-05&     3.37\\
     &80 &1.164E-05&     3.13&           5.223E-06&     3.05&
     5.183E-06&     3.05\\
     &160&1.414E-06&     3.04&           6.480E-07&     3.01&           6.428E-07&     3.01\\
    \hline
    3&20 &7.543E-05&     -&           2.840E-05&     -&           3.058E-05&     -\\
     &40 &8.282E-06&     3.19&           3.208E-06&     3.15&           3.901E-06&     2.97\\
     &80 &8.588E-07&     3.27&           3.501E-07&     3.20&
     4.431E-07&     3.14\\
     &160&8.748E-08&     3.30&           3.642E-08&     3.27&           4.812E-08&     3.20\\
    \hline
    4&20 &2.170E-06&     -&           9.649E-07&     -&           1.752E-06&     -\\
     &40 &3.209E-08&     6.08&           1.746E-08&     5.79&           3.606E-08&     5.60\\
     &80 &8.787E-10&     5.19&           5.031E-10&     5.12&           1.066E-09&     5.08\\
     &160&2.620E-11&     5.07&           1.542E-11&     5.03&
     3.288E-11&     5.02\\
    \hline 
  \end{tabular} 
  \caption{Accuracy test of the SKT population model in Example
    \ref{examp:skt}, with central flux for $\wbxi$ and 
  Lax-Friedrichs flux for $\wFbu$.} 
    \label{tab:skt_i}
\end{table} 

\begin{table}[h!] 
  \centering 
  \begin{tabular}{c|c|c|c|c|c|c|c} 
    \hline   
    $k$&$N$&$L^1$ error& order&$L^2$ error& order&$L^\infty$ error& order\\ 
    \hline 
    1&20 &8.476E-02&     -   &    3.128E-02&    -    & 2.359E-02&     -   \\
     &40 &2.055E-02&     2.04&    7.346E-03&    2.090& 5.189E-03&     2.18\\
     &80 &5.088E-03&     2.01&    1.803E-03&    2.027& 1.235E-03&     2.07\\
     &160&1.268E-03&     2.00&    4.486E-04&    2.007& 3.049E-04&     2.02\\
    \hline 
    2&20 &1.313E-03&     -   &    5.584E-04&    -    & 6.613E-04&     -   \\
     &40 &1.490E-04&     3.14&    6.531E-05&    3.096& 7.530E-05&     3.14\\
     &80 &1.815E-05&     3.04&    8.039E-06&    3.022& 9.240E-06&     3.03\\
     &160&2.250E-06&     3.01&    1.001E-06&    3.005& 1.146E-06&     3.01\\
    \hline 
    3&20 &4.008E-05&     -    &    1.908E-05&    -    & 3.948E-05&     -    \\
     &40 &2.420E-06&     4.05&    1.160E-06&    4.040& 2.645E-06&     3.90\\
     &80 &1.503E-07&     4.01&    7.200E-08&    4.010& 1.682E-07&     3.98\\
     &160&9.374E-09&     4.00&    4.493E-09&    4.002& 1.056E-08&     3.99\\
    \hline 
    4&20 &1.603E-06&     -   &    8.465E-07&    -    & 1.774E-06&     -\\
     &40 &4.754E-08&     5.08&    2.615E-08&    5.017& 6.252E-08&     4.83\\
     &80 &1.469E-09&     5.02&    8.149E-10&    5.004& 2.028E-09&     4.95\\
     &160&4.577E-11&     5.00&    2.544E-11&    5.001& 6.375E-11&     4.99\\
    \hline 
  \end{tabular} 
  \caption{Accuracy test of the SKT population model in Example
    \ref{examp:skt}, with
    alternating fluxes $\wbxi = \bxi_h^-$ and $\wFbu = (F_h \bu_h)^+$.}
    \label{tab:skt_ii}
\end{table} 

\begin{table}[h!] 
  \centering 
  \begin{tabular}{c|c|c|c|c|c|c|c} 
    \hline   
    $\tilde{\alpha}$&$N$&$L^1$ error& order&$L^2$ error& order&$L^\infty$ error& order\\ 
    \hline 
    $0$     &20 &9.502E-05& -    &  3.419E-05& -    &  3.141E-05& -    \\
            &40 &1.196E-05& 2.99&  4.260E-06& 3.00&  4.207E-06& 2.90\\
            &80 &1.501E-06& 2.99&  5.331E-07& 3.00&  5.222E-07& 3.01\\
            &160&1.878E-07& 3.00&  6.667E-08& 3.00&  6.516E-08& 3.00\\
    \hline 
  $10\alpha$&20 &3.941E-05&     -&           1.692E-05&     -&           2.287E-05&     -\\
            &40 &3.833E-06&     3.36&           1.644E-06&     3.36&
            2.245E-06&     3.35\\
            &80 &3.351E-07&     3.52&           1.477E-07&     3.48&           2.136E-07&     3.39\\
            &160&2.720E-08&     3.62&           1.238E-08&     3.58&
            1.865E-08&     3.52\\
    \hline 
$900\alpha$&20 & 2.873E-06&     -&           1.096E-06&     -&           1.981E-06&     -\\
           &40 & 1.558E-07&     4.21&           6.967E-08&     3.98&           1.213E-07&     4.03\\
           &80 & 1.075E-08&     3.86&           4.892E-09&     3.83&
           8.305E-09&     3.87\\
           &160& 7.043E-10&     3.93&           3.261E-10&     3.91&           5.546E-10&     3.90\\

    \hline 
  \end{tabular} 
  \caption{Accuracy test of the SKT population model in Example
    \ref{examp:skt} with $k = 3$, with central flux for $\wbxi$ and the Lax-Friedrichs flux for 
    $\wFbu = \{F_h\bu_h\} +
    \frac{\tilde{\alpha}}{2}[\brho_h]$. Here $\tilde{\alpha} = 0,
  10\alpha, 900\alpha$ respectively.}\label{tab:skt_iii}
\end{table} 
\end{Examp}
\begin{Examp}[Tumor encapsulation]\label{examp:tumor}
  \begin{subequations}\label{eq-1dnum-tumor}
 \begin{empheq}[left=\empheqlbrace]{align}
    \partial_t \rho_1 &= \partial_x\left( 
      (2\rho_1(1-\rho_1)-\beta\gamma\rho_1\rho_2^2)\partial_x\rho_1 
      - 2\beta\rho_1\rho_2(1+\gamma
  \rho_1)\partial_x \rho_2\right),\\
    \partial_t \rho_2 &= 
    \partial_x\left(
      (-2\rho_1\rho_2+\beta\gamma(1-\rho_2)\rho_2^2)\partial_x \rho_1
      + 2\beta\rho_2 (1-\rho_2)(1+\gamma\rho_1)\partial_x\rho_2
    \right).
  \end{empheq}
  \end{subequations}
  System \eqref{eq-1dnum-tumor} comes from the tumor encapsulation model
  proposed by Jackson and Byrne in \cite{jackson2002mechanical}, which describes
  the formation of a dense, fibrous connective tissue surrounding the benign neoplastic
  mass. In the system, $\rho_1$ corresponds to the concentration of tumors, 
  and $\rho_2$ is the concentration of the surrounding tissue. In
  \cite{jungel2012entropy}, the authors pointed out 
  the system can be formulated as a gradient flow with the entropy 
  \[E = \int_\Omega \rho_1(\log \rho_1 - 1) + \rho_2(\log \rho_2 -1) +
  (1-\rho_1-\rho_2)(\log(1-\rho_1-\rho_2)-1) dx,\]
  if $0 \leq \gamma<4/\sqrt{\beta}$. 
  Then $\bxi = (\log\frac{\rho_1}{1-\rho_1-\rho_2},
  \log\frac{\rho_2}{1-\rho_1-\rho_2})^T$ and the corresponding $F$ is
  semi-positive definite with the prescribed parameters.
  In the numerical test, we firstly consider the case $\beta = 0.0075$ and
  $\gamma = 10$. The initial condition is set as 
  \[\rho_1(x,0) = \frac{1}{8}\left(1+\tanh\frac{0.1-x}{0.05}\right),\qquad
  \rho_2(x,0) = \frac{1}{8}\left(1-\tanh\frac{0.1-x}{0.05}\right).\]
  The zero-flux boundary condition is applied in the simulation. The computational
  domain is set as $\Omega = [0,1]$ with $h = 0.02$ and $h = 0.04$
  respectively. We choose the time step to be $\tau = 0.02h^2$. The numerical
  results are given in Figure \ref{fig-tumorG}.
  We then consider problems with a strong cell-induced pressure, with $\beta =
  0.0075$ and $\gamma = 1000$. These parameters are rescaled from
\cite{jackson2002mechanical} and we drop the source term in their simulation.
In the numerical test, $0.95\theta^\veps_{l,i}$ instead of $\theta^\veps_{l,i}$
is used in the scaling limiter for robustness and all other settings are the
same. The numerical results are given in Figure \ref{fig-tumor}.
Although the system may no long possess a decaying entropy since $\gamma >
4/\sqrt{\beta}$, it seems that our numerical method still produces
satisfying results and captures the sharp leading front of $\rho_2$. 
  
  \begin{figure}
    \centering
    \begin{subfigure}[h!]{0.45\textwidth}
      \centering
      \includegraphics[width=\textwidth]{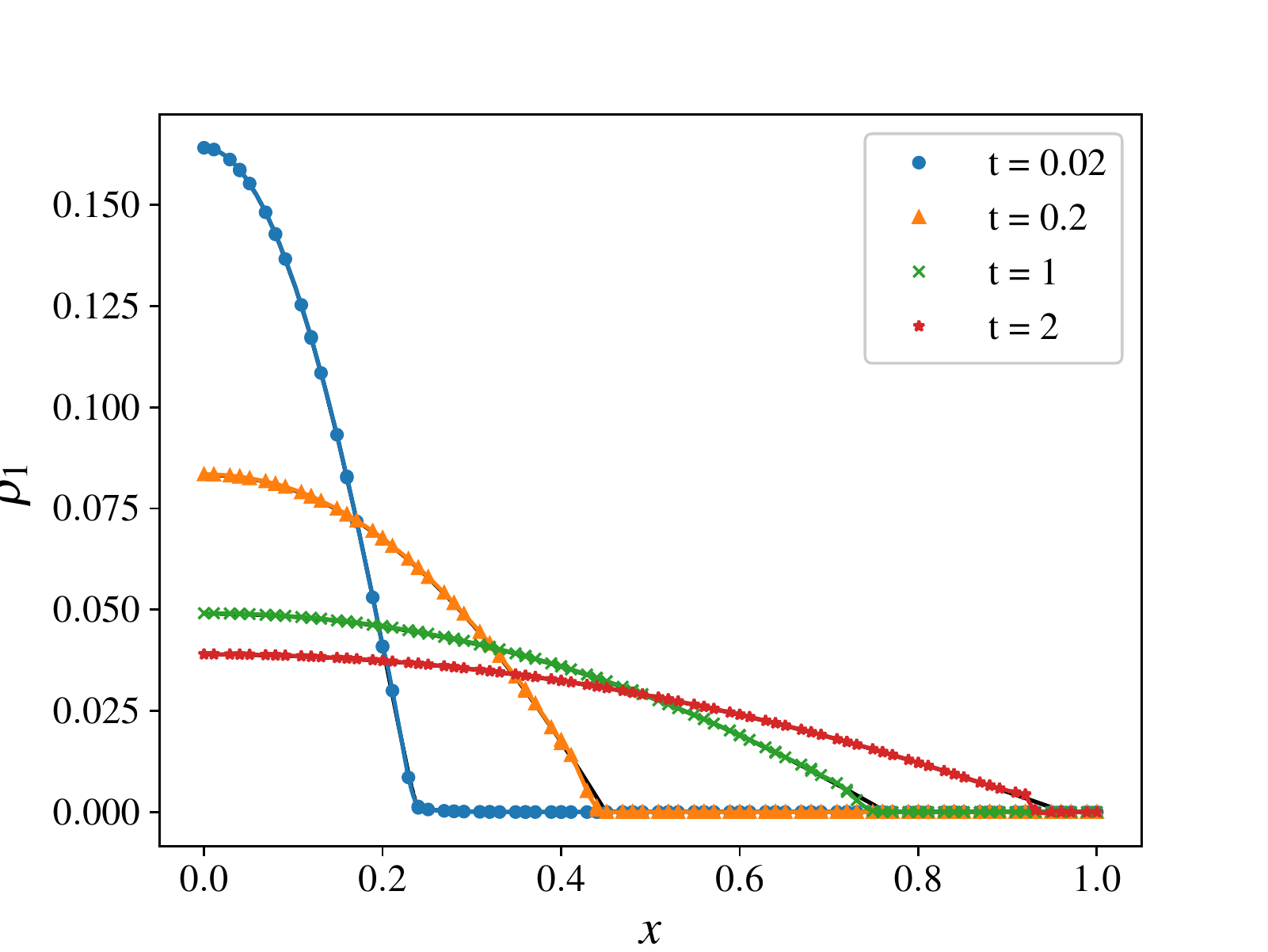}
      \caption{$\rho_1$, $h = 0.04$.}\label{subfig-tumorG-a1}
    \end{subfigure}
    ~
    \begin{subfigure}[h!]{0.45\textwidth}
      \centering
      \includegraphics[width=\textwidth]{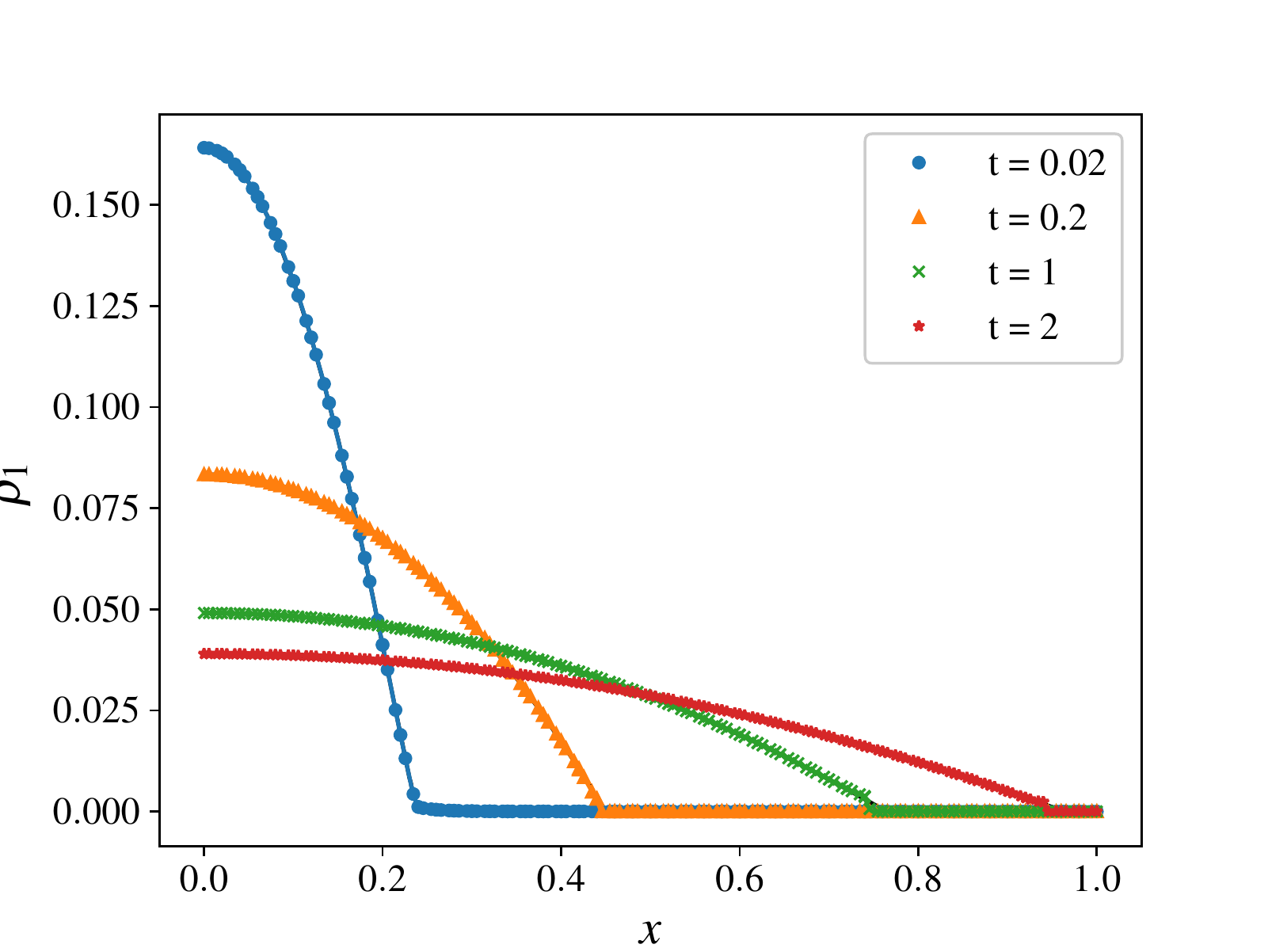}
      \caption{$\rho_1$, $h = 0.02$.}\label{subfig-tumorG-b1}
    \end{subfigure}
    ~
    \begin{subfigure}[h!]{0.45\textwidth}
      \centering
      \includegraphics[width=\textwidth]{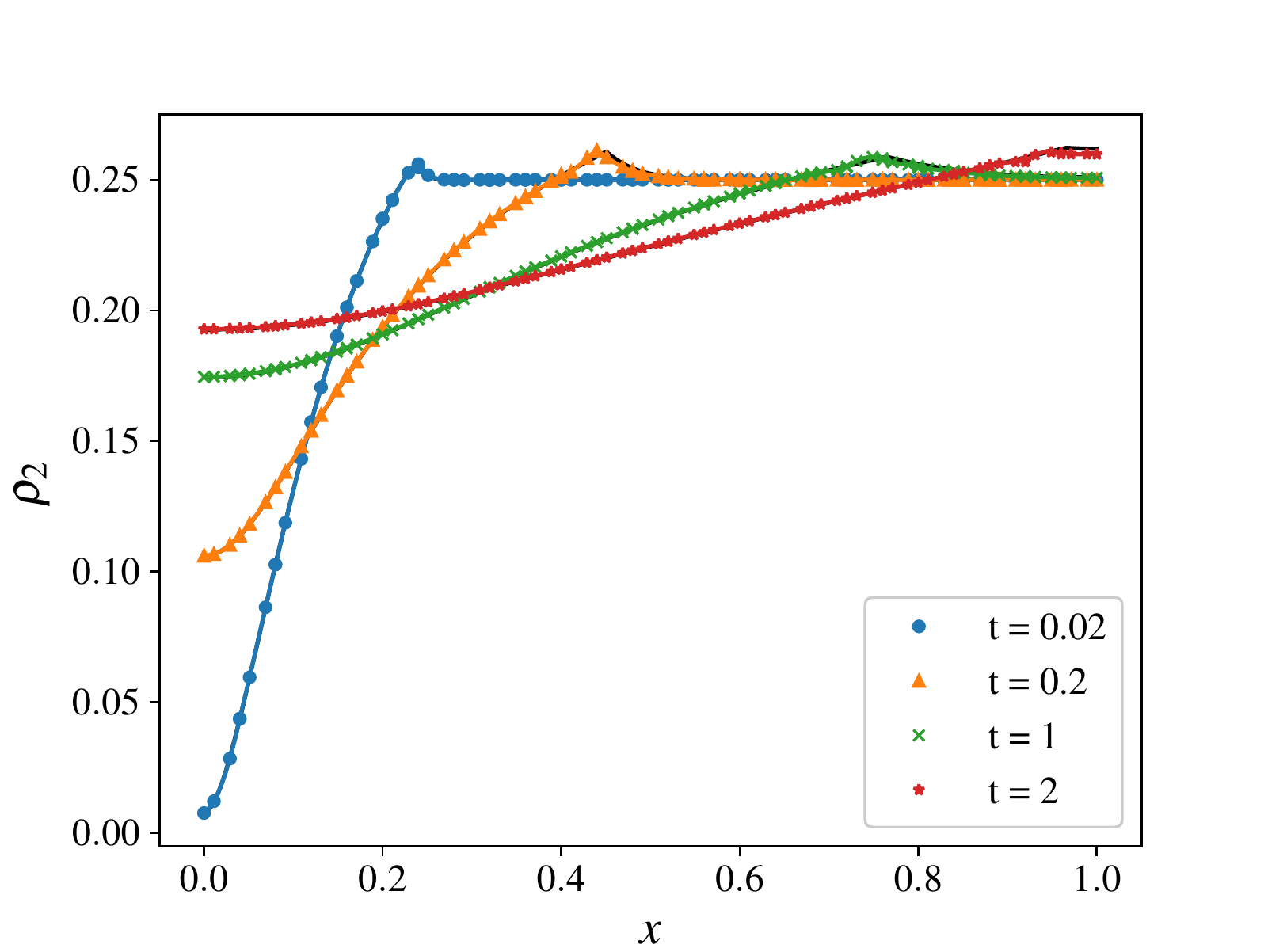}
      \caption{$\rho_2$, $h = 0.04$.}\label{subfig-tumorG-a2}
    \end{subfigure}
    ~
    \begin{subfigure}[h!]{0.45\textwidth}
      \centering
      \includegraphics[width=\textwidth]{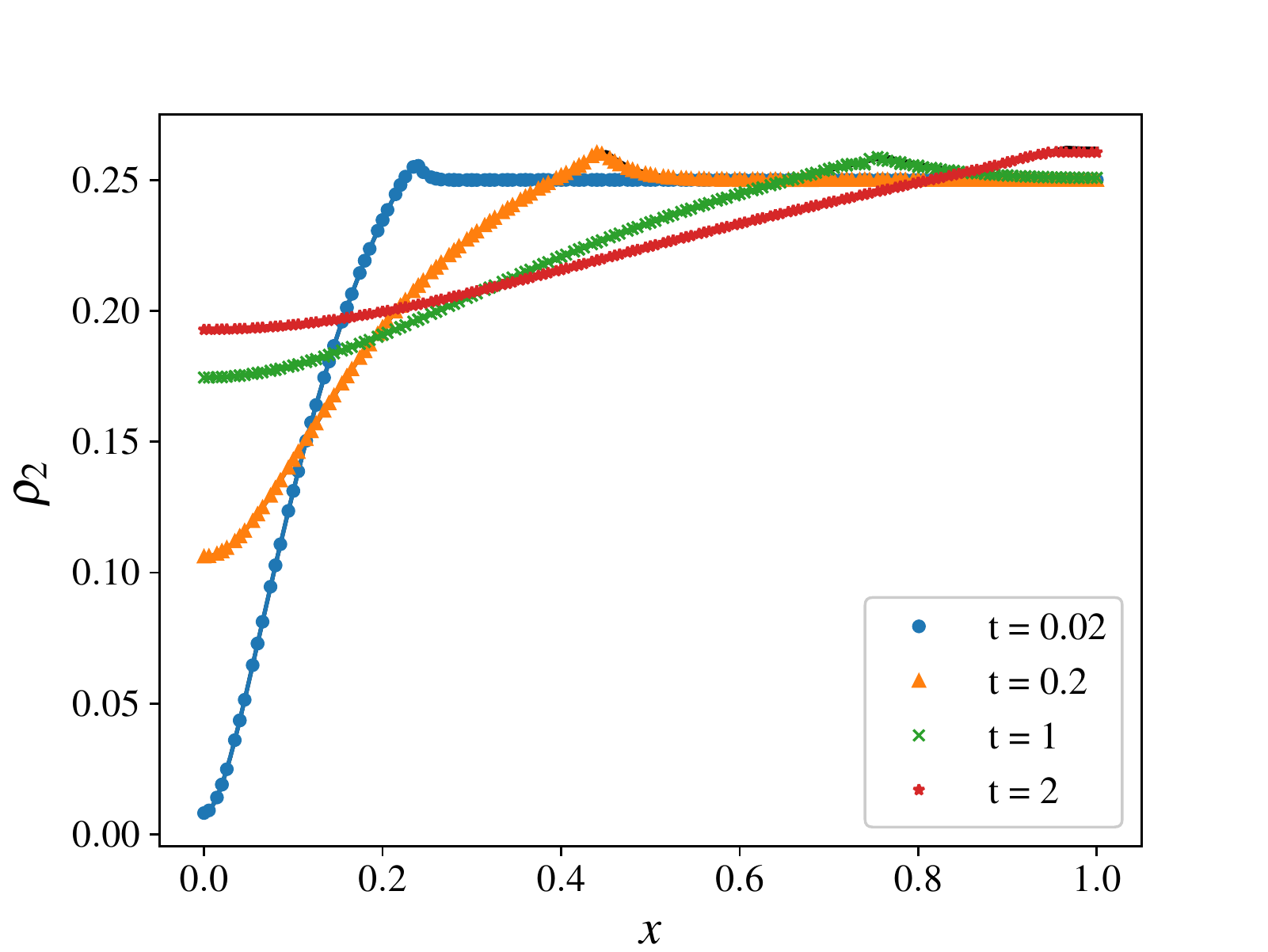}
      \caption{$\rho_2$, $h = 0.02$.}\label{subfig-tumorG-b2}
    \end{subfigure}
    \begin{subfigure}[h!]{0.45\textwidth}
      \centering
      \includegraphics[width=\textwidth]{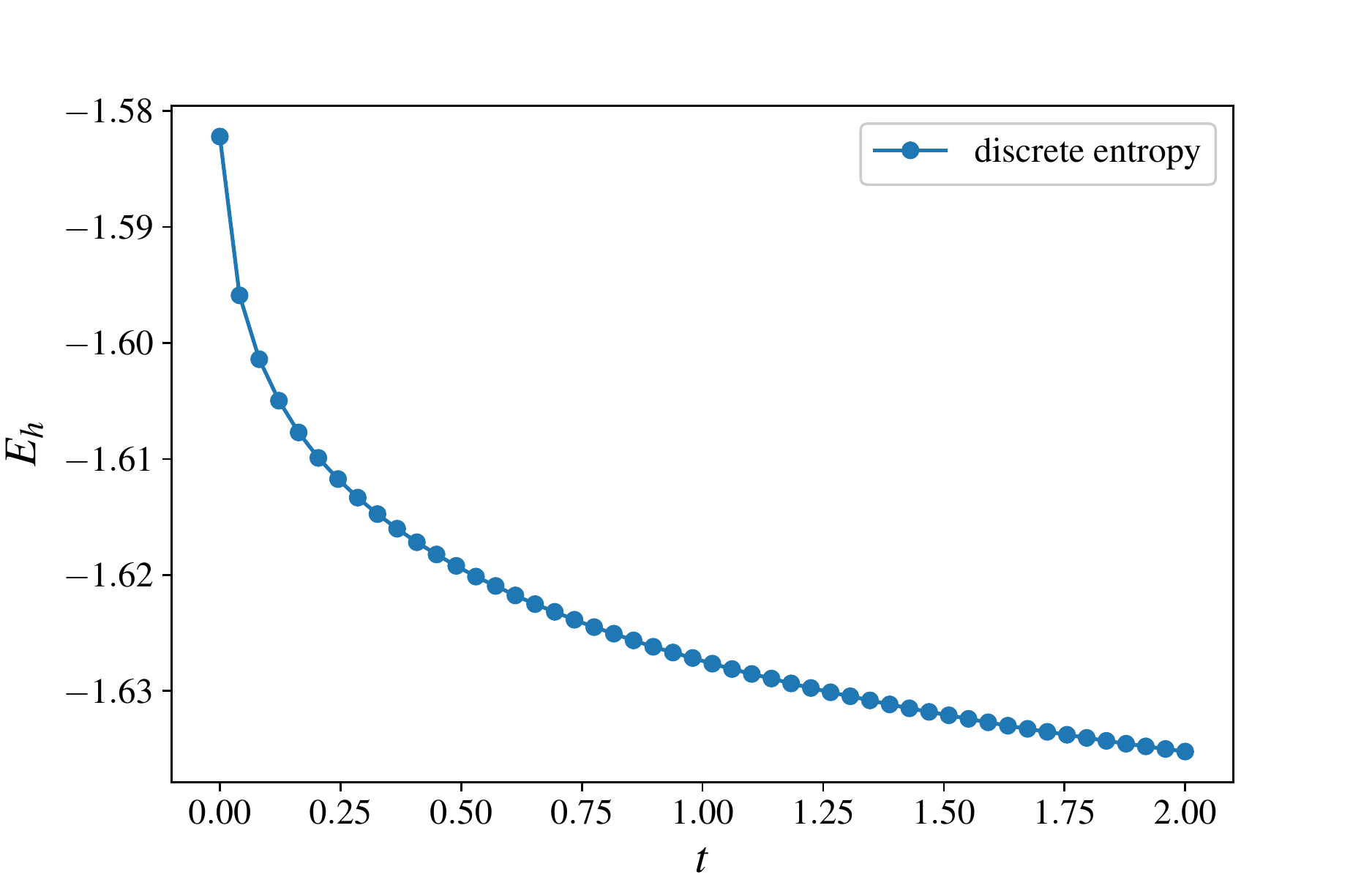}
      \caption{$E_h$, $h = 0.04$.}\label{subfig-tumorG-c1}
    \end{subfigure}
    ~
    \begin{subfigure}[h!]{0.45\textwidth}
      \centering
      \includegraphics[width=\textwidth]{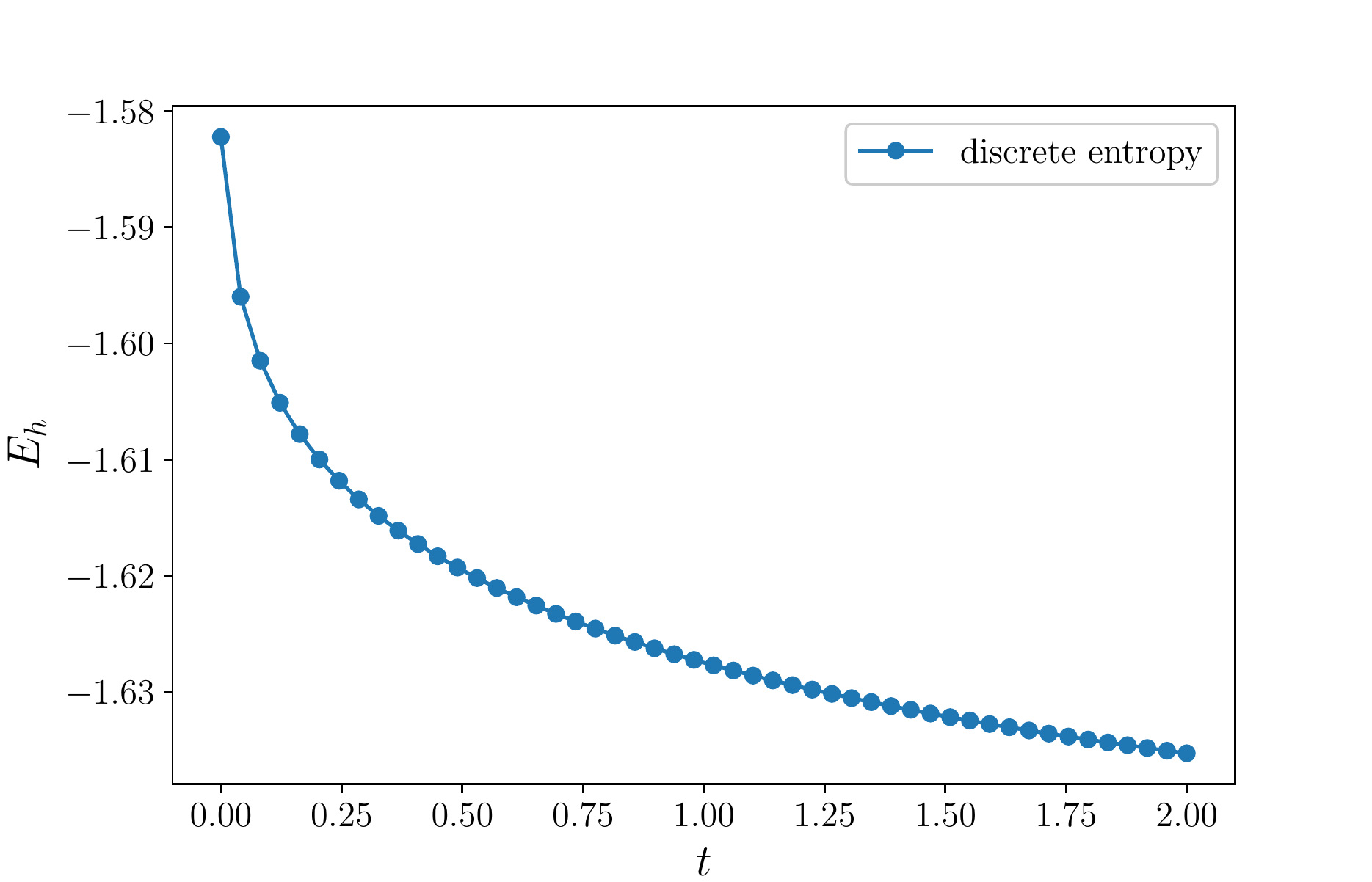}
      \caption{$E_h$, $h = 0.02$.}\label{subfig-tumorG-c2}
    \end{subfigure}
    \caption{Numerical solutions to the tumor encapsulation problem
      in Example \ref{examp:tumor} with $\beta = 0.0075$ and $\gamma = 10$
      at $t = 0.02$, $t = 0.2$, $t = 1$ and $t = 2$. We use piecewise cubic
      polynomials in the scheme. The mesh size is $h = 0.04$ in Figure 
      \ref{subfig-tumorG-a1} and Figure \ref{subfig-tumorG-a2}, and is
      $h = 0.02$ in Figure \ref{subfig-tumorG-b1} and Figure
      \ref{subfig-tumorG-b2}. The corresponding entropy profiles are given in 
      Figure \ref{subfig-tumorG-c1} and Figure\ref{subfig-tumorG-c2}.
The reference solutions are given in black lines obtained by using the
$P^1$ scheme on a mesh with $h = 0.002$.}\label{fig-tumorG}
  \end{figure}

  \begin{figure}
    \centering
    \begin{subfigure}[h!]{0.45\textwidth}
      \centering
      \includegraphics[width=\textwidth]{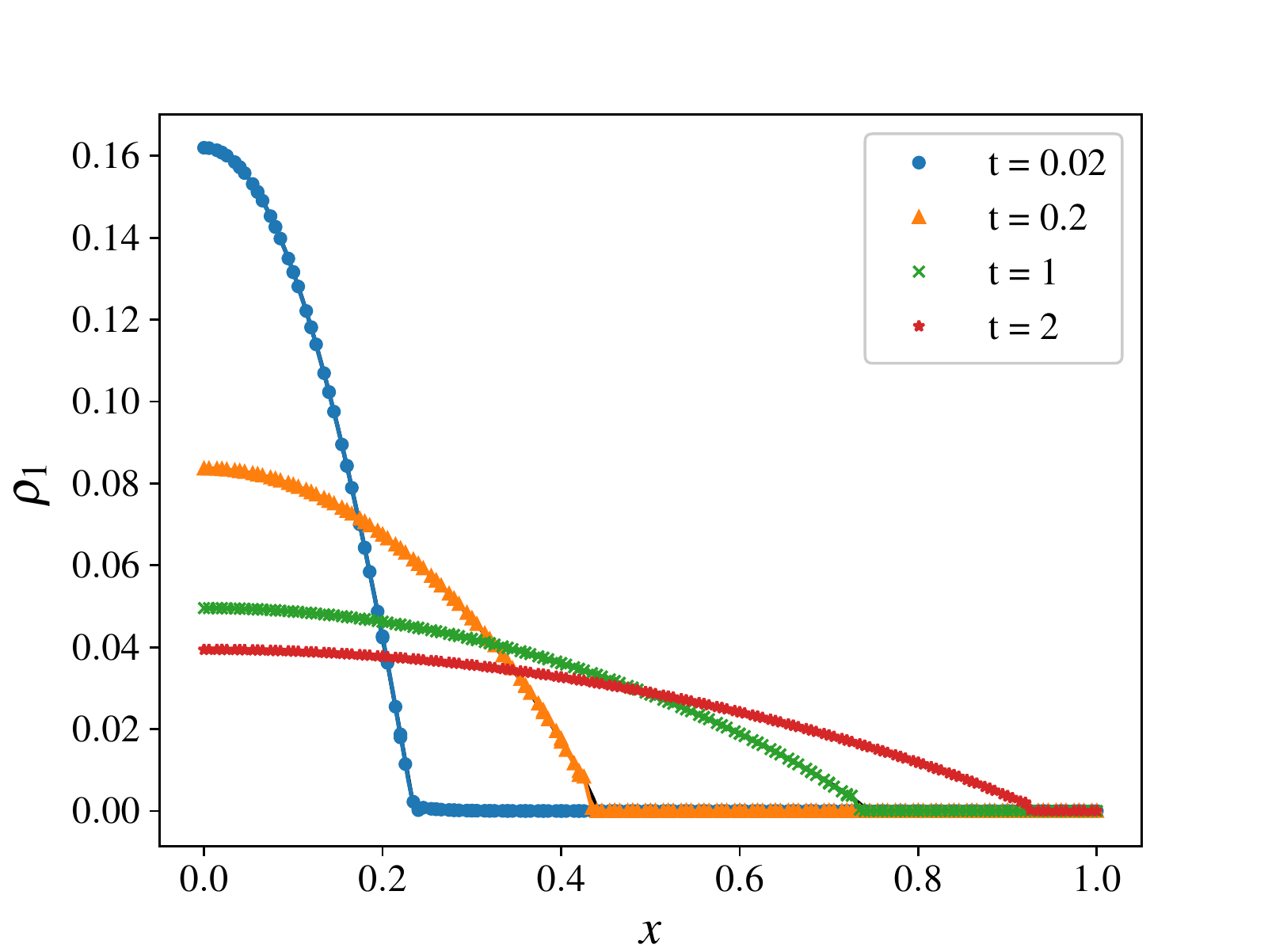}
      \caption{$\rho_1$, $k = 3$.}\label{subfig-tumor-a1}
    \end{subfigure}
    ~
    \begin{subfigure}[h!]{0.45\textwidth}
      \centering
      \includegraphics[width=\textwidth]{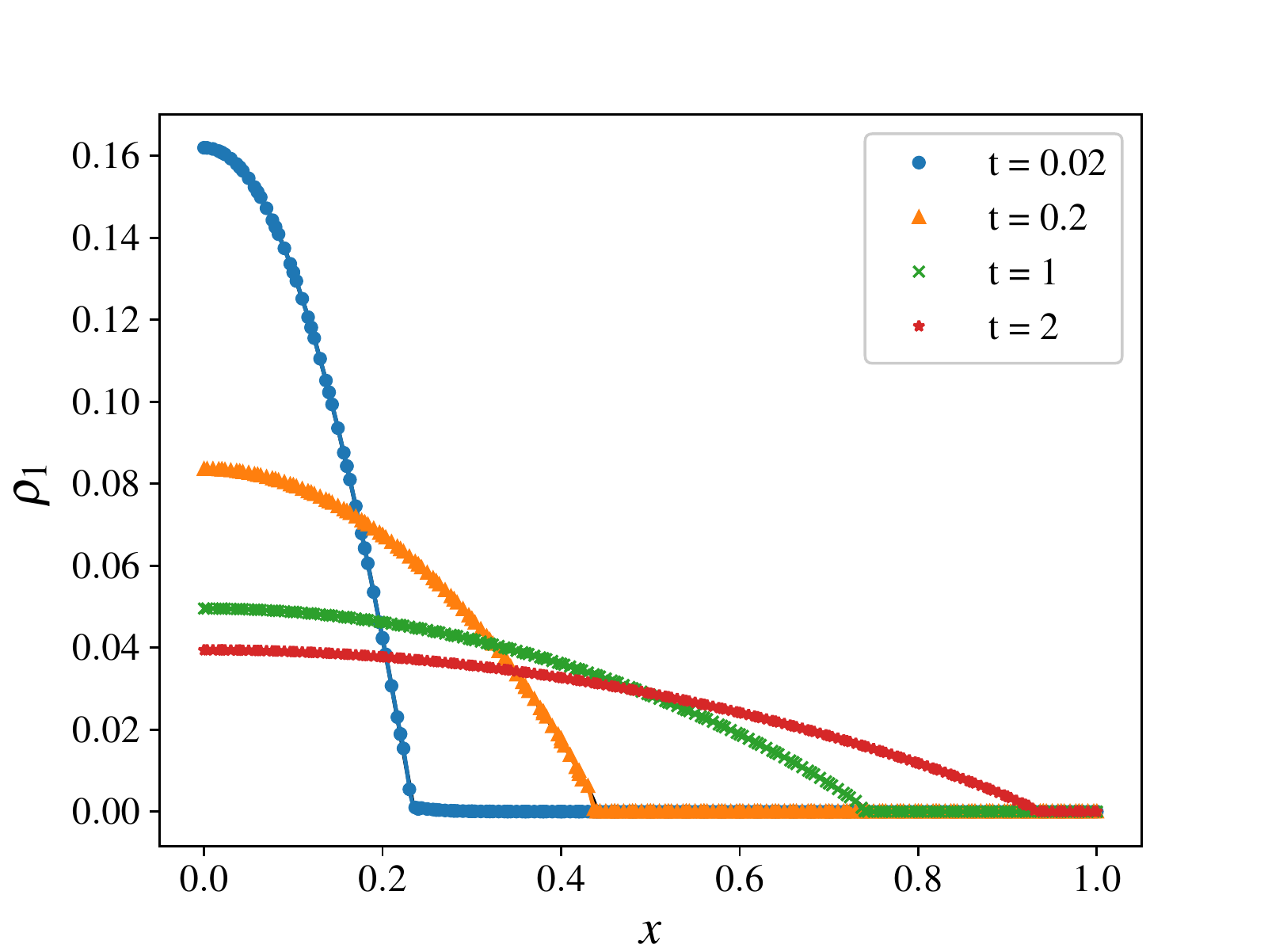}
      \caption{$\rho_1$, $k = 4$.}\label{subfig-tumor-b1}
    \end{subfigure}
    ~
    \begin{subfigure}[h!]{0.45\textwidth}
      \centering
      \includegraphics[width=\textwidth]{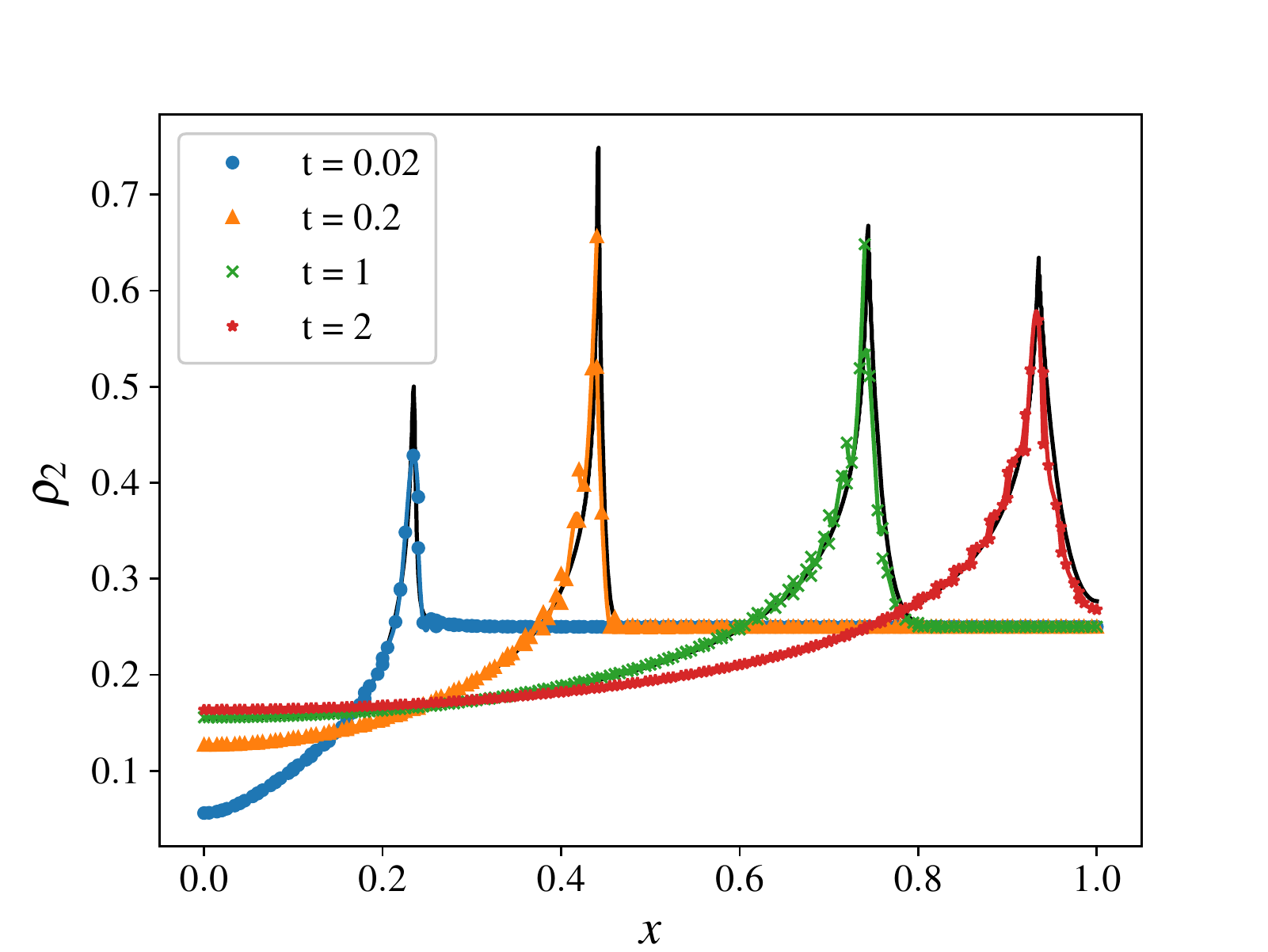}
      \caption{$\rho_2$, $k = 3$.}\label{subfig-tumor-a2}
    \end{subfigure}
    ~
    \begin{subfigure}[h!]{0.45\textwidth}
      \centering
      \includegraphics[width=\textwidth]{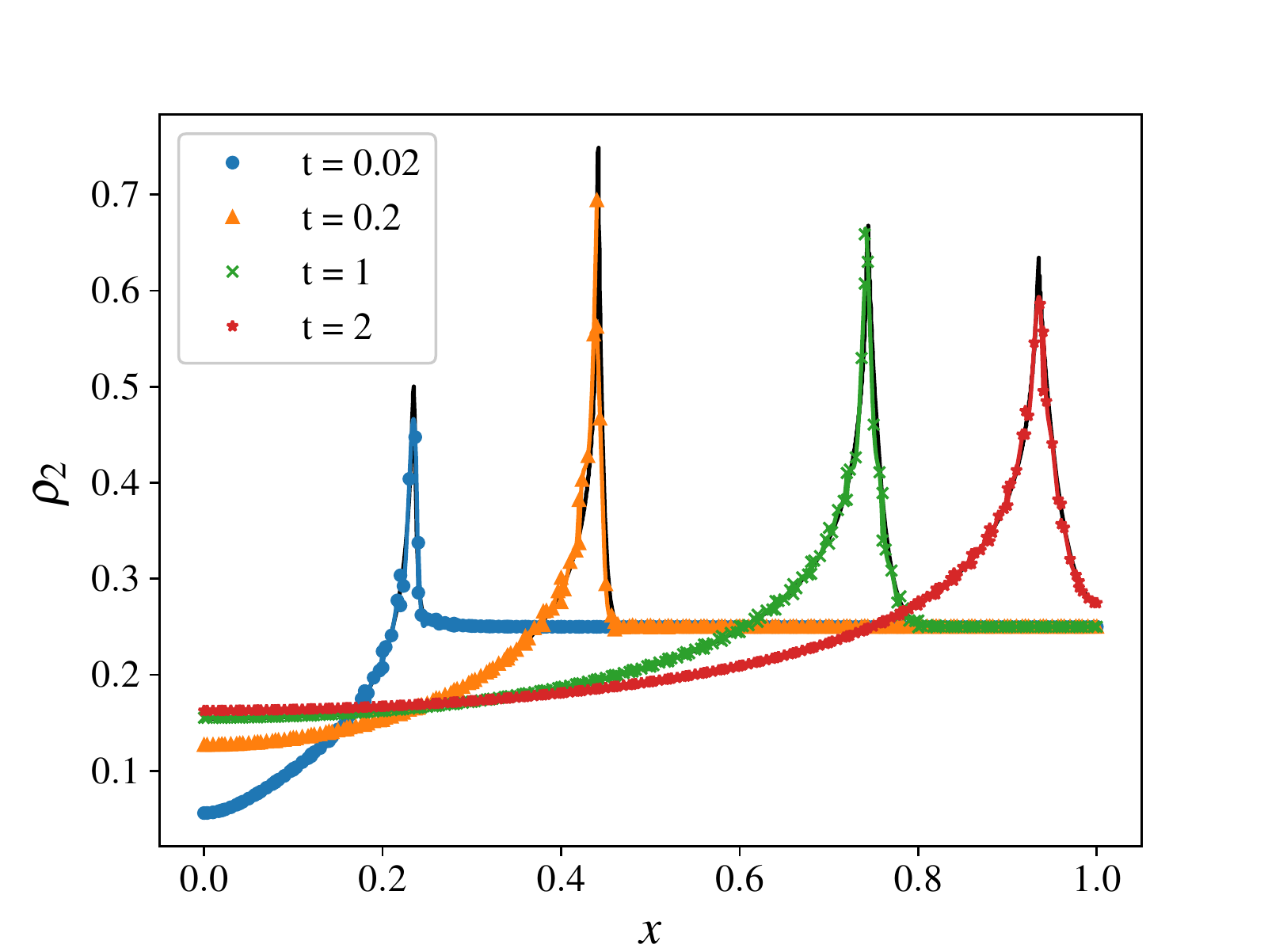}
      \caption{$\rho_2$, $k = 4$.}\label{subfig-tumor-b2}
    \end{subfigure}
    \caption{Numerical solutions to the tumor encapsulation problem
      in Example \ref{examp:tumor} with $\beta = 0.0075$ 
      and $\gamma = 1000$ at $t = 0.02$, $t = 0.2$, $t =1$ and $t = 2$. 
      The mesh size is $h = 0.02$ and the time step is $\tau = 0.02h^2$. The
      scaling parameter is set as $0.95\theta^\veps_{l,i}$ in the
      positivity-preserving procedure. 
      Piecewise cubic polynomials are used in Figure \ref{subfig-tumor-a1} and
      Figure \ref{subfig-tumor-a2}, and piecewise quartic polynomials are used
      in Figure \ref{subfig-tumor-b1} and Figure \ref{subfig-tumor-b2}. 
      The reference solution is given in black lines obtained by using the
    $P^1$ scheme on a mesh with $h = 0.001$.}\label{fig-tumor}
  \end{figure}

\end{Examp}
\begin{Examp}[Surfactant spreading]\label{examp:surf}
  We perform numerical simulations of a system
  modelling the surfactant spreading on a thin viscous film, which can be used
  for analyzing the delivery of aerosol for curing the respiratory distress
  syndrome. The system was
  derived by Jensen and Grotberg in \cite{jensen1992insoluble} and was then
  analyzed by Escher et al. 
  in \cite{escher2011global}. In this system, $\rho_1$ represents the film thickness and
  $\rho_2$ is the concentration of the surfactant. The parameter $g$ corresponds to a
  gravitational force.
  \begin{subequations}\label{eq-1dnum-surf}
 \begin{empheq}[left=\empheqlbrace]{align}
      \partial_t \rho_1 &=  \partial_x \left(\frac{g}{3}\rho_1^3\partial_x \rho_1
        +\frac{1}{2}\rho_1^2 \partial_x
      \rho_2\right),\\
      \partial_t \rho_2 &= \partial_x\left(\frac{g}{2}\rho_1^2\rho_2\partial_x\rho_1
      +\rho_1\rho_2 \partial_x \rho_2\right).
    \end{empheq}
 \end{subequations}
  Following the derivation in \cite{escher2011global}, the system is associated with the following entropy functional
\[E = \int_\Omega \frac{g}{2}\rho_1^2 + \rho_2(\log \rho_2 -1) dx.\]
Accordingly, $\bxi = (g\rho_1,\log \rho_2)^T$, $F =
\diag(\brho)\left(\begin{matrix}\frac{1}{3}\rho_1^2&\frac{1}{2}\rho_1\rho_2\\
\frac{1}{2}\rho_1^2&\rho_1\rho_2\end{matrix}\right)$ and $\bz \cdot F\bz =
\frac{1}{12}\rho_1^3z_1^2+\rho_1(\frac{1}{2}\rho_1z_1 + \rho_2z_2)^2\geq 0$.
It is shown in \cite{jensen1992insoluble} that the similarity solutions to
\eqref{eq-1dnum-surf}
would develop shocks as $g = 0$. 
In this numerical test, we consider a weak gravitational effect with $g = 0.02$.
A uniform mesh with $h = 0.05$ is used on the spatial domain $[0,3]$. The time
step is set as $\tau = 0.02 h^2$. We apply the zero-flux boundary condition and assume
the initial condition to be $\rho_1 = 0.5$ and $\rho_2 =
0.5(1-\tanh\frac{x-0.5}{0.1})$. Numerical solutions obtained with $k = 3$ and $k =
4$ are given in Figure \ref{fig-surf}. Under the same setting without applying
positivity-preserving limiter, the solutions blow up shortly after $t = 0.1718$
for $k = 3$ and $t = 0.1691$ for $k = 4$.

  \begin{figure}
    \centering
    \begin{subfigure}[h!]{0.45\textwidth}
      \centering
      \includegraphics[width=\textwidth]{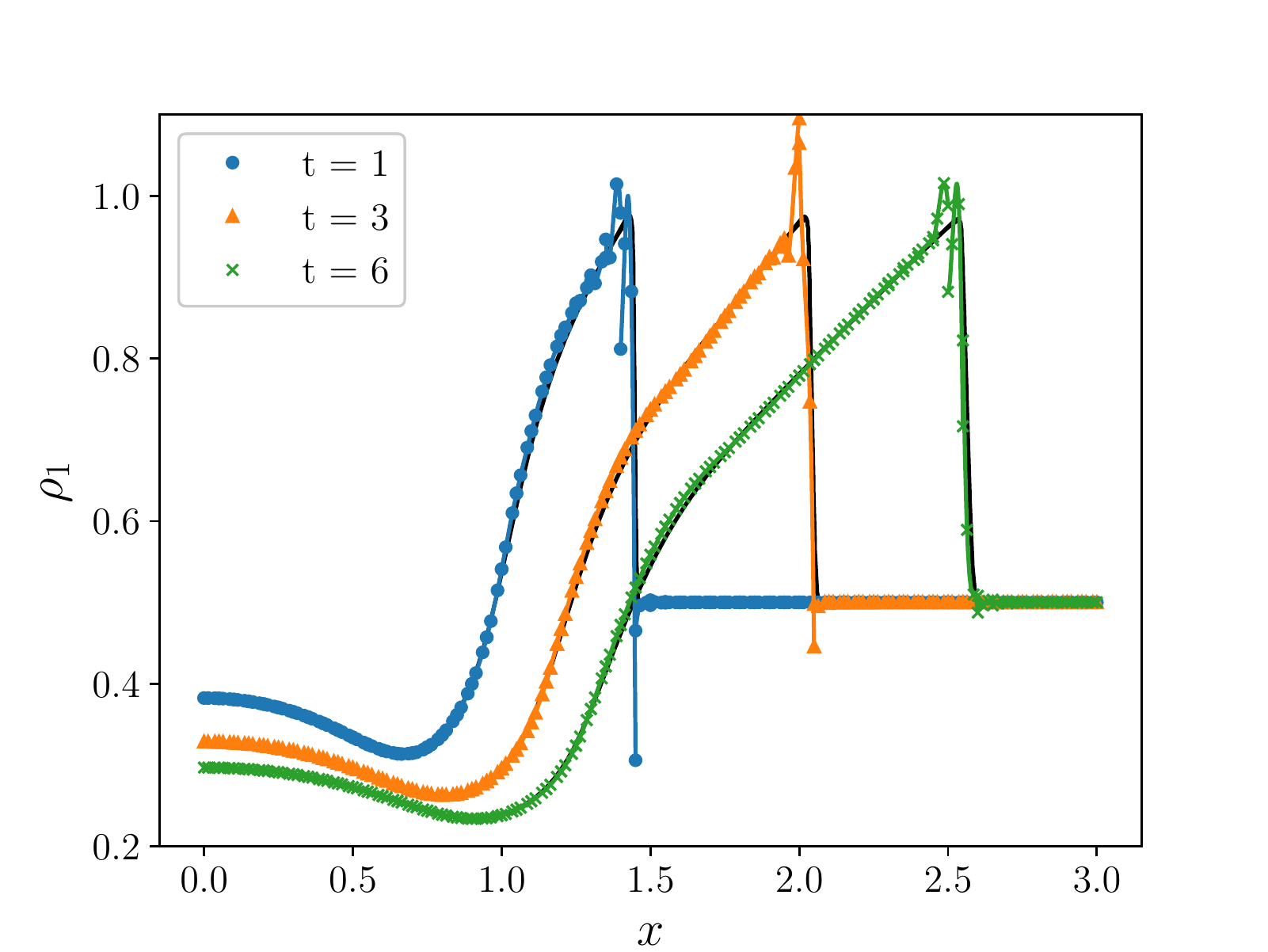}
      \caption{$\rho_1$, $k = 3$.}\label{subfig-surf-a1}
    \end{subfigure}
    ~
    \begin{subfigure}[h!]{0.45\textwidth}
      \centering
      \includegraphics[width=\textwidth]{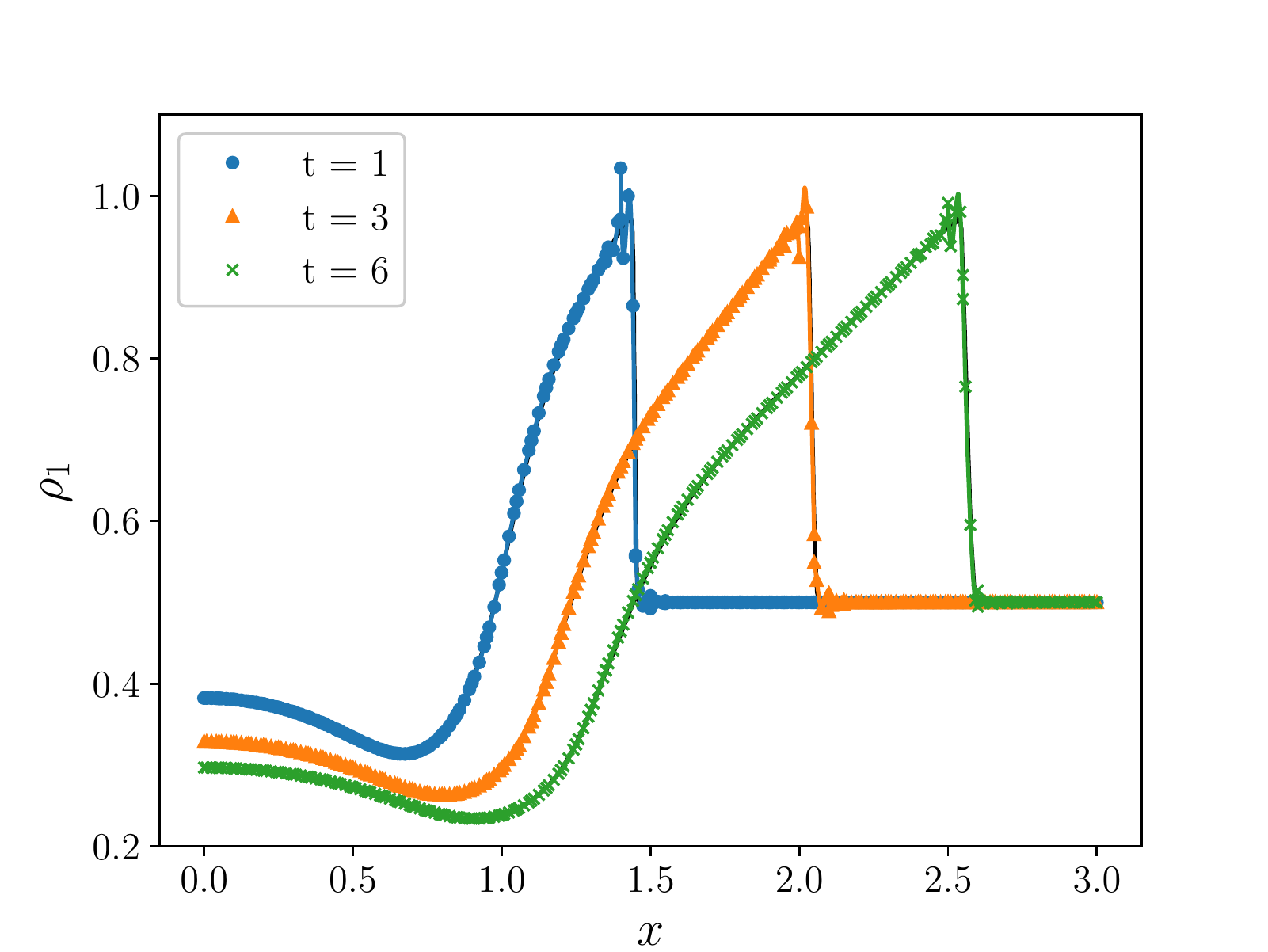}
      \caption{$\rho_1$, $k = 4$.}\label{subfig-surf-b1}
    \end{subfigure}
    ~
    \begin{subfigure}[h!]{0.45\textwidth}
      \centering
      \includegraphics[width=\textwidth]{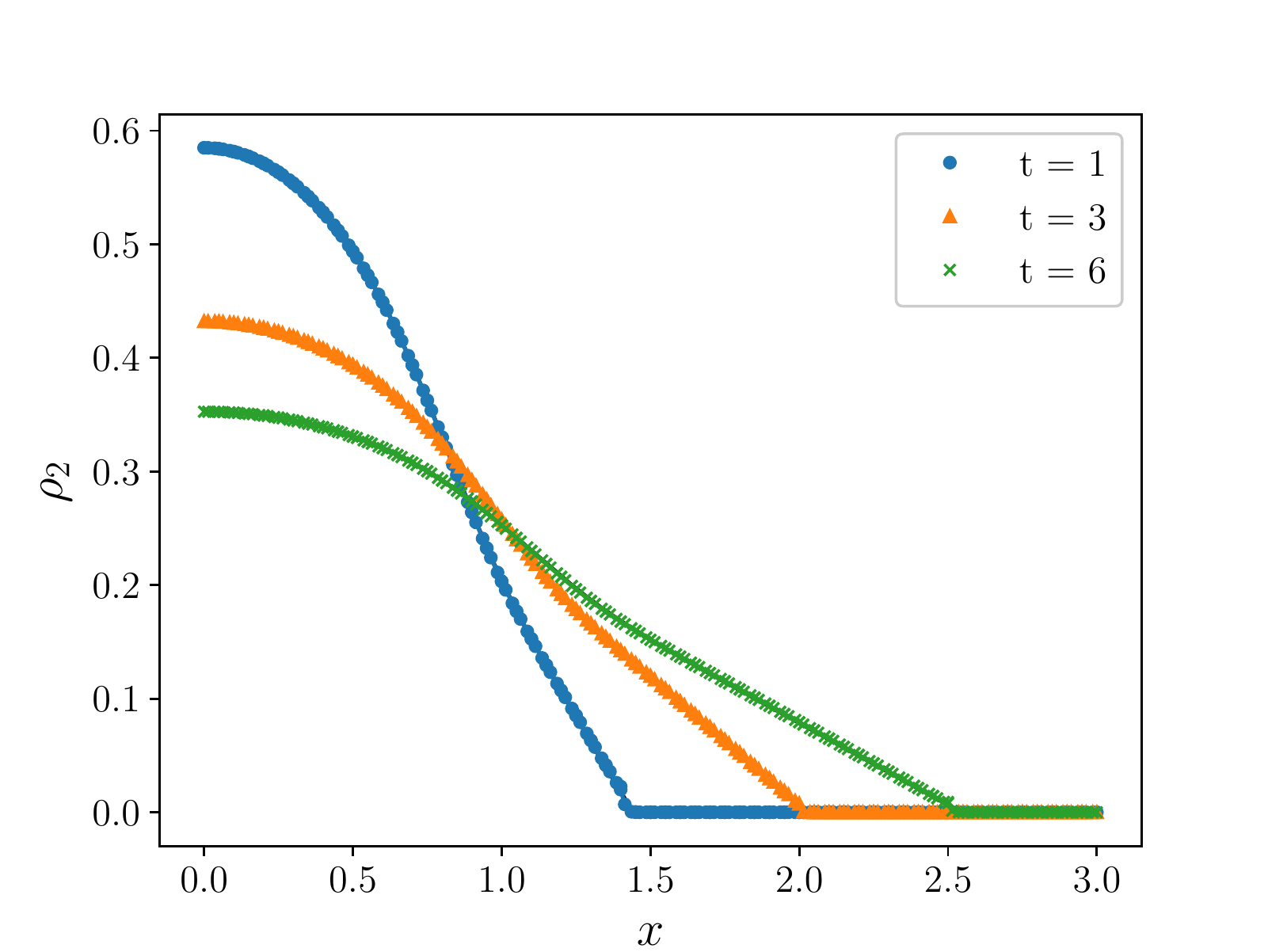}
      \caption{$\rho_2$, $k = 3$.}\label{subfig-surf-a2}
    \end{subfigure}
    ~
    \begin{subfigure}[h!]{0.45\textwidth}
      \centering
      \includegraphics[width=\textwidth]{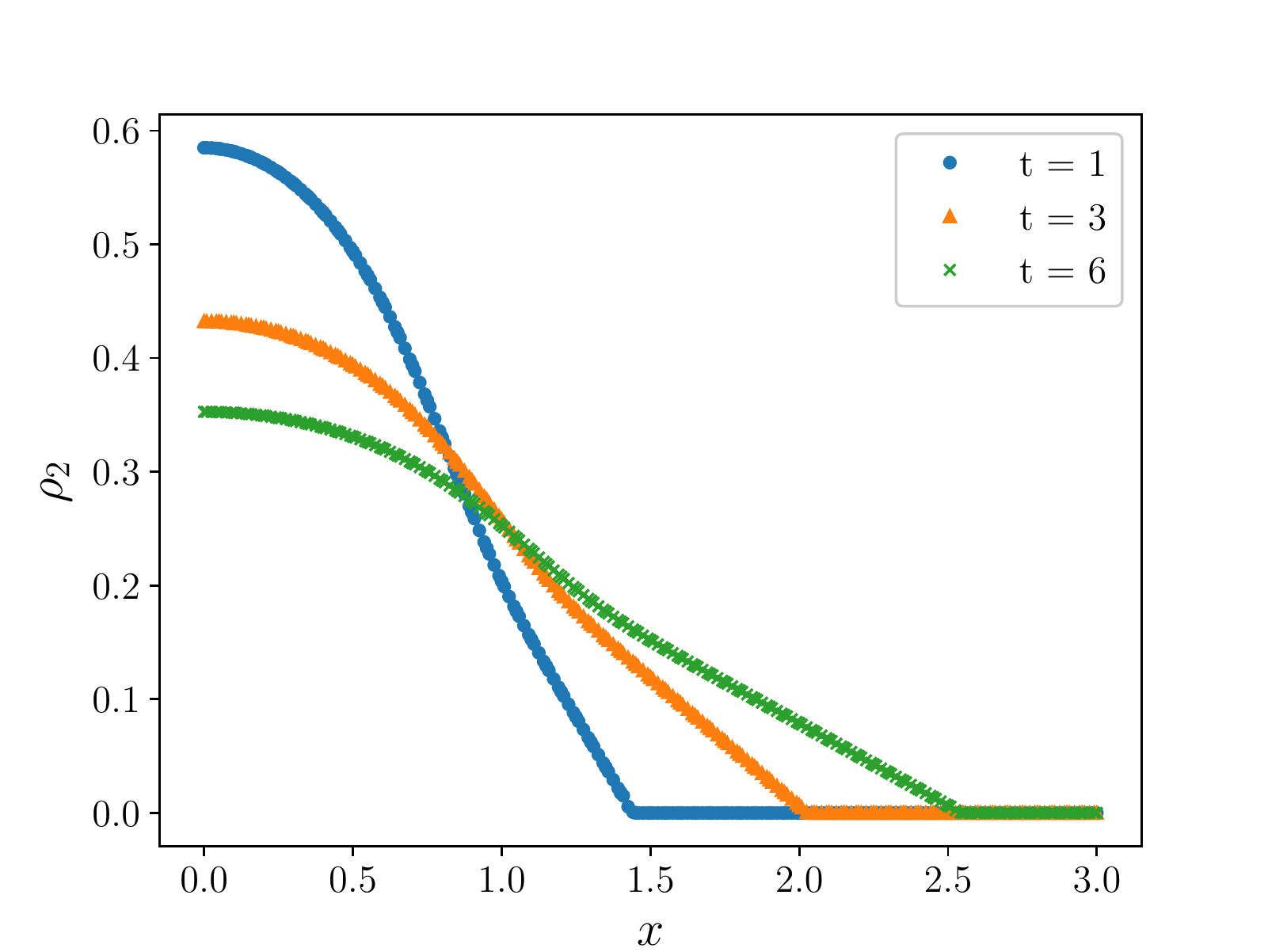}
      \caption{$\rho_2$, $k = 4$.}\label{subfig-surf-b2}
    \end{subfigure}
    ~
    \begin{subfigure}[h!]{0.45\textwidth}
      \centering
      \includegraphics[width=\textwidth]{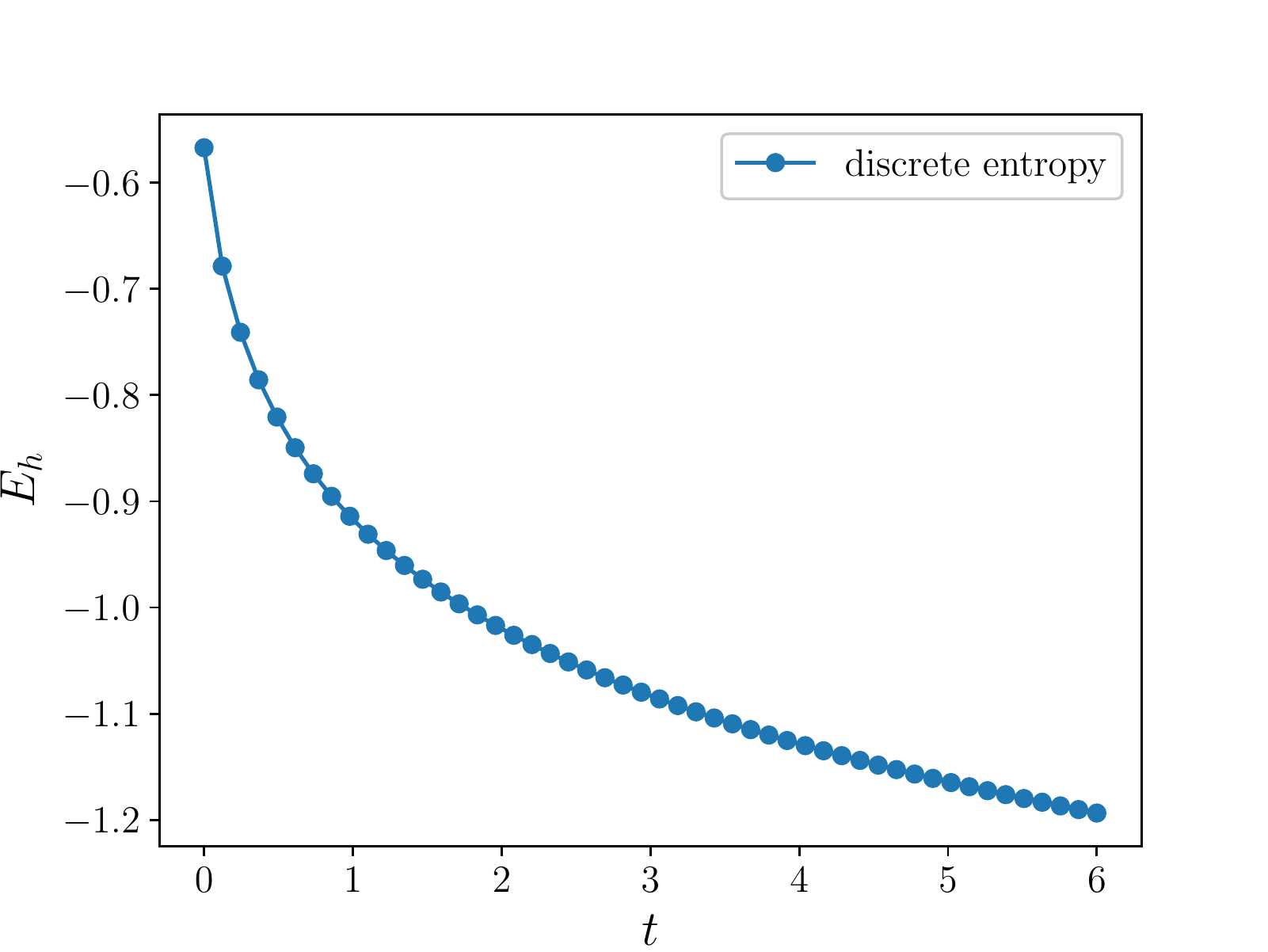}
      \caption{$E_h$, $k = 3$.}\label{subfig-surf-a3}
    \end{subfigure}
    ~
    \begin{subfigure}[h!]{0.45\textwidth}
      \centering
      \includegraphics[width=\textwidth]{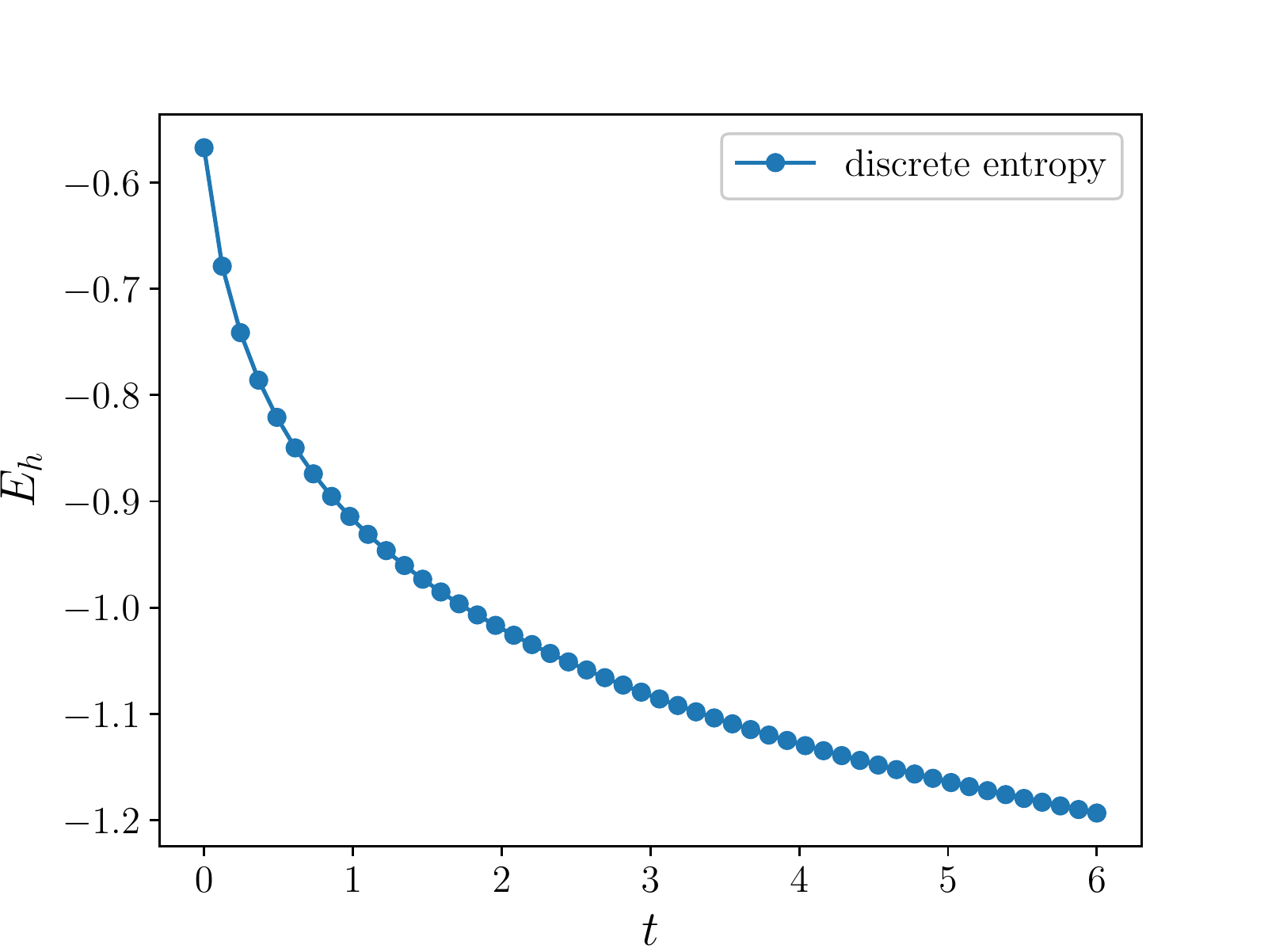}
      \caption{$E_h$, $k = 4$.}\label{subfig-surf-b3}
    \end{subfigure}
    \caption{Numerical solutions to the surfactant spreading problem in Example
    \ref{examp:surf} at $t = 1$, $t = 3$ and $t = 6$. 
    The mesh size is $h = 0.05$ and the time step is $\tau =
  0.02h^2$. We apply piecewise cubic polynomials for producing Figure \ref{subfig-surf-a1} and Figure \ref{subfig-surf-a2}. Piecewise quartic polynomials are used in
Figure \ref{subfig-surf-b1} and Figure \ref{subfig-surf-b2}.
Positivity-preserving limiter is activated mainly near the leading front of
$\rho_2$. The reference solutions are given in black lines
obtained with $P^1$ scheme on a mesh with $h = 0.0025$.
The profiles of discrete entropy are depicted in Figure
\ref{subfig-surf-a3} and
Figure \ref{subfig-surf-b3}. 
}\label{fig-surf}
\end{figure}
\end{Examp}
\section{Two-dimensional numerical tests}\label{sec-tests2}
\setcounter{equation}{0}
\setcounter{figure}{0}
\setcounter{table}{0}
In this section, we provide several numerical examples of two-dimensional problems on Cartesian
meshes. The accuracy test is given in Example \ref{examp:2dskt}. Then we 
apply the numerical scheme to solve the two-dimensional surfactant spreading
problem in Example \ref{examp:2dsurf} and the seawater intrusion model in
Example \ref{examp:seawater}. 
\begin{Examp}\label{examp:2dskt}
  We start with the accuracy test and consider the two-dimensional version of Example \ref{examp:skt}
  with a source term $\bs = (s_1,s_2)^T$. 
  \begin{equation*}\left\{
  \begin{aligned}
    \partial_t \rho_1 =\nabla\cdot \left( (2\rho_1 + \rho_2)\nabla \rho_1 +
    \rho_1\nabla\rho_2\right) + s_1,\\
    \partial_t \rho_2 = \nabla\cdot \left(\rho_2 \nabla \rho_1 +
    (\rho_1+2\rho_2)\nabla\rho_2\right)+s_2.
\end{aligned}\right.
  \end{equation*}
  The problem is again associated with the logarithm entropy.
  \[E = \iint_\Omega \rho_1(\log \rho_1 -1) + \rho_2(\log \rho_2 -1) dxdy.\]
  Then $\bxi = (\log \rho_1, \log \rho_2)^T$ and $F = \diag(\brho)\left(
  \begin{matrix}2\rho_1+\rho_2&\rho_2\\\rho_1&2\rho_2+\rho_1\end{matrix}\right)$.
  We assume the exact solution $\brho = (\rho_1,\rho_2)^T=
  (0.5\sin(\pi(x+y+t))+1,0.5\cos(\pi(x-y-0.5t)) + 1)^T$. The source term
  $\bs = (s_1,s_2)^T$
  is computed accordingly. We compute to $t = 0.03$. The time step is set as
  $\tau = 0.0003h^2$ for $k = 1,2,3$ and $\tau = 0.0001h^2$ for $k = 4$.  
\begin{table}[h!] 
  \centering 
  \begin{tabular}{c|c|c|c|c|c|c|c} 
    \hline   
    $k$&$N^x$&$L^1$ error& order&$L^2$ error& order&$L^\infty$ error& order\\ 
    \hline 
    1&10 & 1.185E-01&	    - &5.349E-02&		  - &5.732E-02&	    -\\ 
     &20 & 4.656E-02&	1.35	&2.147E-02&	1.32	&2.265E-02&	1.34\\
     &40 & 1.773E-02&	1.39	&8.266E-03&	1.38	&8.692E-03&	1.38\\
     &80&  6.218E-03&	1.51	&2.921E-03&	1.50	&3.084E-03&	1.50\\
    \hline
    2&10 & 8.206E-03&	-     &4.362E-03&	-     &8.402E-03&	-    \\
     &20 & 9.124E-04&	3.17	&5.659E-04&	2.95	&9.942E-04&	3.08\\
     &40 & 1.077E-04&	3.08	&7.186E-05&	2.98	&1.213E-04&	3.04\\
     &80&  1.315E-05&	3.03	&9.068E-06&	2.99	&1.515E-05&	3.00\\
    \hline 
    3&10 & 9.481E-04&	-     &5.256E-04&	-     &9.990E-04&	-    \\
     &20 & 1.061E-04&	3.16	&5.847E-05&	3.17	&1.112E-04&	3.17\\
     &40 & 1.128E-05&	3.23	&6.169E-06&	3.25	&1.216E-05&	3.19\\
     &80&  1.119E-06&	3.33	&6.042E-07&	3.35	&1.232E-06&	3.30\\
    \hline
    4&10 &4.685E-05&  -    & 2.649E-05& -    & 7.459E-05&  -\\
     &20 &1.212E-06&  5.27& 8.159E-07& 5.02& 2.837E-06&  4.72\\
     &40 &3.328E-08&  5.19& 2.273E-08& 5.17& 7.754E-08&  5.19\\
    \hline 
  \end{tabular} 
  \caption{Accuracy test of the cross-diffusion system in Example
    \ref{examp:2dskt}, with central flux for $\wbxi$ and 
  Lax-Friedrichs flux for $\wFbu$.} 
    \label{tab:2dskt_i}
\end{table} 
\begin{table}[h!] 
  \centering 
  \begin{tabular}{c|c|c|c|c|c|c|c} 
    \hline   
    $k$&$N^x$&$L^1$ error& order&$L^2$ error& order&$L^\infty$ error& order\\ 
    \hline 
    1&10 &3.848E-01&     -    & 1.853E-01&     -    & 2.415E-01&     -\\ 
     &20 &9.608E-02&     2.00& 4.397E-02&     2.08& 4.372E-02&     2.47\\ 
     &40 &2.376E-02&     2.02& 1.075E-02&     2.03& 1.014E-02&     2.11\\ 
     &80& 5.916E-03&     2.01& 2.667E-03&     2.01& 2.525E-03&     2.01\\ 
    \hline
    2&10 &1.510E-02&     -    & 8.132E-03&     -    & 2.121E-02&     -\\ 
     &20 &1.673E-03&     3.17& 9.563E-04&     3.09& 2.328E-03&     3.19\\ 
     &40 &1.923E-04&     3.12& 1.169E-04&     3.03& 2.662E-04&     3.13\\ 
     &80& 2.294E-05&     3.07& 1.453E-05&     3.01& 3.192E-05&     3.06\\ 
    \hline 
    3&10 &8.253E-04&     -    & 4.525E-04&     -    & 1.449E-03&     -    \\
     &20 &4.998E-05&     4.05& 2.922E-05&     3.95& 1.044E-04&     3.80\\
     &40 &3.075E-06&     4.02& 1.850E-06&     3.98& 6.675E-06&     3.97\\
     &80 &1.912E-07&     4.01& 1.161E-07&     4.00& 4.216E-07&     3.99\\
    \hline
    4&10 &5.174E-05&     -     & 3.287E-05&     -     & 1.448E-04&     -\\
     &20 &1.683E-06&     4.94& 1.127E-06&     4.87& 4.972E-06&     4.86\\
     &40 &5.276E-08&     5.00& 3.620E-08&     4.96& 1.485E-07&     5.07\\
    \hline 
  \end{tabular} 
  \caption{Accuracy test of the cross-diffusion system in Example
    \ref{examp:2dskt} with alternating fluxes
  $\wbxi = \bxi_h^-$ and $\wFbu = (F_h \bu_{h})^+$.}
    \label{tab:2dskt_ib}
\end{table} 
\begin{table}[h!] 
  \centering 
  \begin{tabular}{c|c|c|c|c|c|c|c} 
    \hline   
    $\tilde{\alpha}$&$N^x$&$L^1$ error& order&$L^2$ error& order&$L^\infty$ error& order\\ 
    \hline 
    0&10 &1.082E-03&     -    &     6.170E-04&      -    &   1.139E-03&   -\\               
      &20& 1.328E-04&      3.07&    7.682E-05&      3.01&  1.373E-04&    3.05\\          
      &40& 1.651E-05&      3.01&    9.595E-06&      3.00&  1.722E-05&    3.00\\          
      &80& 2.063E-06&      3.00&    1.199E-06&      3.00&  2.157E-06&    3.00\\
    \hline 
$100\alpha$&10&2.505E-04&  -    &    1.266E-04&      -& 1.528E-04&      -   \\
           &20&1.234E-05&  4.34&    6.425E-06&      4.30&  1.234E-05&      3.63\\
           &40&6.709E-07&  4.20&    3.609E-07&      4.15&  9.094E-07&      3.76\\ 
           &80&3.923E-08&  4.10&    2.178E-08&      4.05&  6.480E-08&      3.81\\ 
    \hline 
  \end{tabular} 
  \caption{Accuracy test for Example \ref{examp:2dskt}, with central flux for
    $\wbxi$ and Lax-Friedrichs flux for $\wFbu = \{F_h\bu_h\}
    + \frac{\tilde{\alpha}}{2}[\brho_h]$. Here $\tilde{\alpha} =
0,100\alpha$.}\label{tab:2dskt_iii}
\end{table} 
\end{Examp}
\begin{Examp}\label{examp:2dsurf}
  In this numerical example, we simulate the two-dimensional surfactant
  spreading.
  \begin{subequations}\label{eq-2dnum-surf}
   \begin{empheq}[left=\empheqlbrace]{align}
      \partial_t \rho_1 &=  \nabla\cdot \left(\frac{g}{3}\rho_1^3\nabla \rho_1
        +\frac{1}{2}\rho_1^2 \nabla
      \rho_2\right),\\
      \partial_t \rho_2 &=
      \nabla\cdot\left(\frac{g}{2}\rho_1^2\rho_2\nabla\rho_1
      +\rho_1\rho_2 \nabla \rho_2\right).
    \end{empheq}
  \end{subequations}
    Again $\rho_1$ and $\rho_2$ correspond to the film thickness and surfactant
    concentration respectively. The associated entropy is 
    \[E = \iint_\Omega \frac{g}{2}\rho_1^2 + \rho_2(\log\rho_2-1) dxdy.\]
    As before, $\bxi = (g\rho_1, \log\rho_2)^T$ and $F(\brho) =
    \diag(\brho)\left(\begin{matrix} \frac{1}{3}\rho_1^2 &
      \frac{1}{2}\rho_1\rho_2\\\frac{1}{2}\rho_1^2 & \rho_1\rho_2
  \end{matrix}\right)$.
    We assume zero-flux boundary condition on $\Omega = [0,2]\times[0,2]$ and
    compute to $t = 0.25$ with $h^x=h^y = 0.02$ and $\tau
    = 0.003(h^x)^2$. The gravitational coefficient is set as $g = 0.001$. 
    Numerical results are given in Figure \ref{fig-2dsurf}.
    The numerical method does capture the sharp transition of the bowl-shaped
    leading front of $\rho_1$, though with oscillations.
    \begin{figure}[h!]
    \centering
    \begin{subfigure}[h!]{0.45\textwidth}
      \centering
      \includegraphics[width=\textwidth]{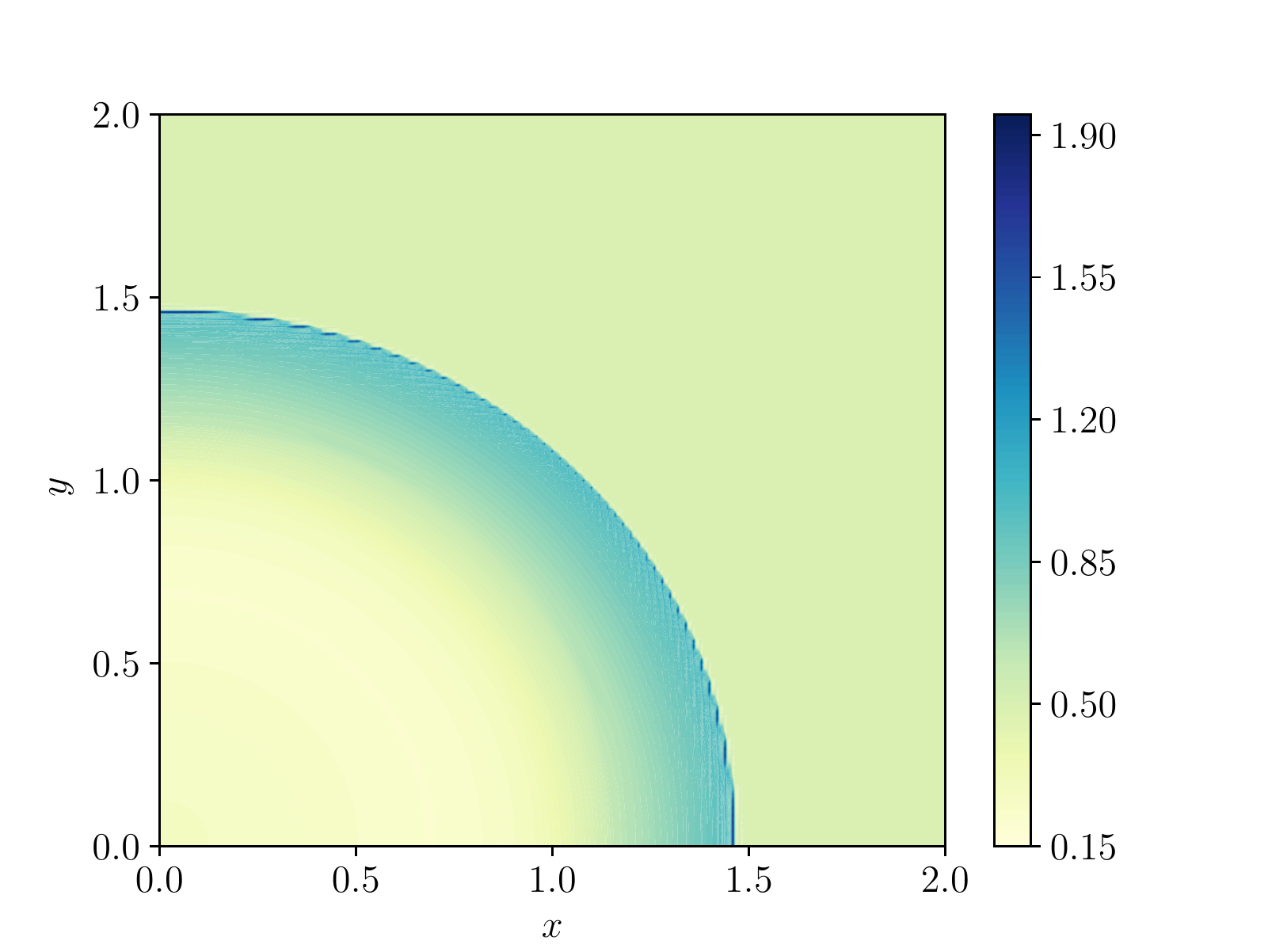}
      \caption{$\rho_1$, $k = 3$.}\label{subfig-surf2d-b1}
    \end{subfigure}
    ~
    \begin{subfigure}[h!]{0.45\textwidth}
      \centering
      \includegraphics[width=\textwidth]{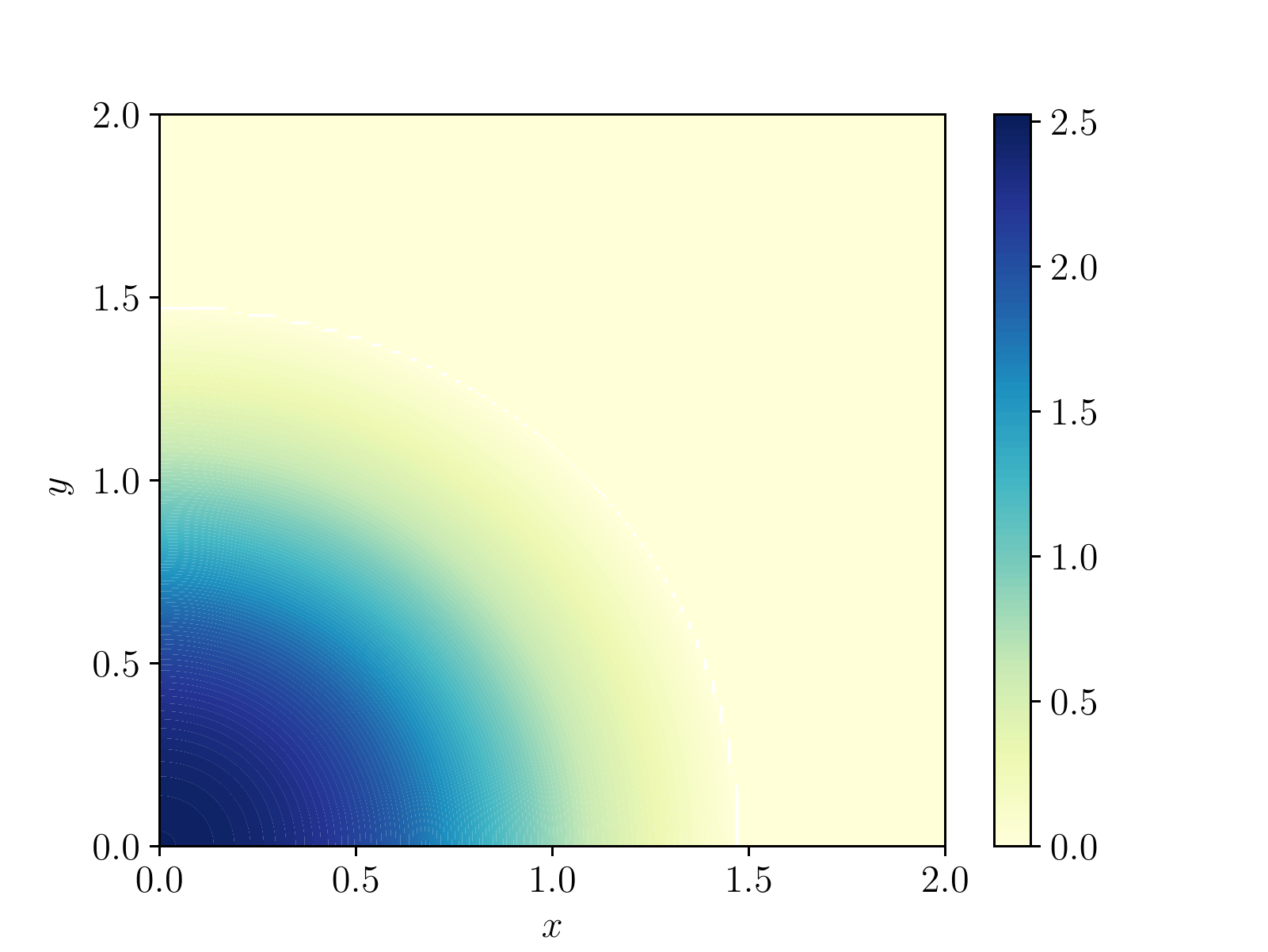}
      \caption{$\rho_2$, $k = 3$.}\label{subfig-surf2d-b2}
    \end{subfigure}
    ~
    \begin{subfigure}[h!]{0.45\textwidth}
      \centering
      \includegraphics[width=\textwidth]{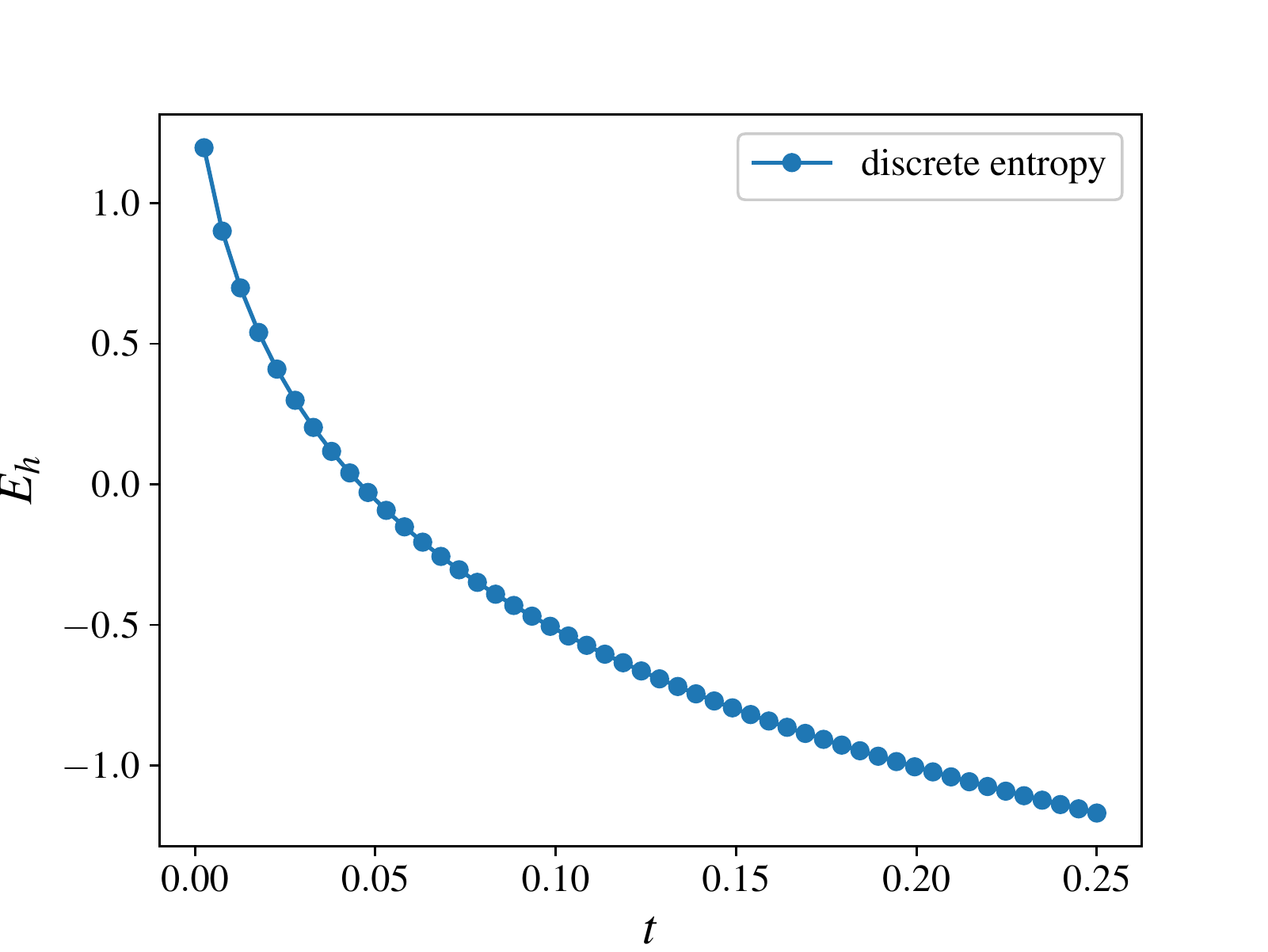}
      \caption{$E_h$, $k = 3$.}\label{subfig-surf2d-b3}
    \end{subfigure}
    \caption{Numerical solutions to the surfactant spreading problem
      \eqref{eq-2dnum-surf} at $t = 0.25$. 
      The solution is computed with piecewise cubic polynomials 
      $k = 3$, mesh size $h = 0.02$ and time step $\tau =
      0.003h^2$.}\label{fig-2dsurf}
\end{figure}
\end{Examp}

\begin{Examp}[Seawater intrusion]\label{examp:seawater}
  This test case is taken from \cite{oulhaj2017finite} on
  solving a cross-diffusion system modelling the seawater intrusion in an
  unconfined aquifer. 
  \begin{subequations}\label{eq-seawater}
 \begin{empheq}[left=\empheqlbrace]{align}
    \partial_t \rho_1 &= \nabla\cdot\left(\mu \rho_1\nabla(\rho_1+\rho_2+b)\right),\\
    \partial_t \rho_2 &= \nabla\cdot\left(
    \rho_2\nabla(\mu\rho_1+\rho_2+b)\right).
\end{empheq}
\end{subequations}
Here $z = b(x,y)$ gives the impermeable interface between the seawater and the
bedrock. The saltwater sits on the bedrock between $z = b(x,y)$ and $z = b(x,y) +
\rho_2(x,y,t)$. From $z = b(x,y) + \rho_2(x,y,t)$ to $z= b(x,y) + \rho_2(x,y,t) + 
\rho_1(x,y,t)$ is the freshwater, which is immersible with the saltwater. The
parameter $\mu \in (0,1)$ is the mass density ratio between the freshwater and
the saltwater. \eqref{eq-seawater} is associated with the energy functional 
\[E = \iint_\Omega \frac{\mu}{2}(\rho_1 + \rho_2 + b)^2 + \frac{1-\mu}{2}(\rho_2 +
b)^2 dxdy.\]
It can be show that there exists a non-negative solution to \eqref{eq-seawater}
satisfying the energy decay property
\[\frac{d}{dt}E = - \iint_\Omega \mu^2 \rho_1|\nabla(\rho_1+\rho_2+b)|^2 +
\rho_2|\nabla(\mu \rho_1 + \rho_2 + b)|^2 dxdy.\]
We assume the zero-flux boundary condition and use a $20\times 20$ square mesh
on $[0,1]\times [0,1]$. We compute to $t =
12$ with $\tau = 0.002(h^x)^2$. The parameter is set as $\mu = 0.9$. The initial
condition $\brho_0 = (\rho_1^0, \rho_2^0)^T$, with
\[\rho_1^0(x,y) = \left\{\begin{matrix} 0.5 &\text{ if } x\leq0.25\\0 &\text{
otherwise}\end{matrix}\right.\]
and \[\rho_2^0(x,y) =
\left\{\begin{matrix} b(0.5,0)-b(x,y) - (x-0.5) &\text{ if } x\leq 0.5\\
0&\text{ otherwise}\end{matrix}\right.,\] is used in our numerical simulation.
Here the seabed function is 
\[b(x,y) =
\max(0,0.5(1-16(x-0.5)^2)(\cos(\pi y) + 2)).\]
Numerical solutions at different time are depicted in Figure \ref{fig-seawater}.
One can see that solutions converge to a steady state as $t$ becomes large. 
  \begin{figure}
    \centering
    \begin{subfigure}[h!]{0.48\textwidth}
      \centering
      \includegraphics[width=\textwidth]{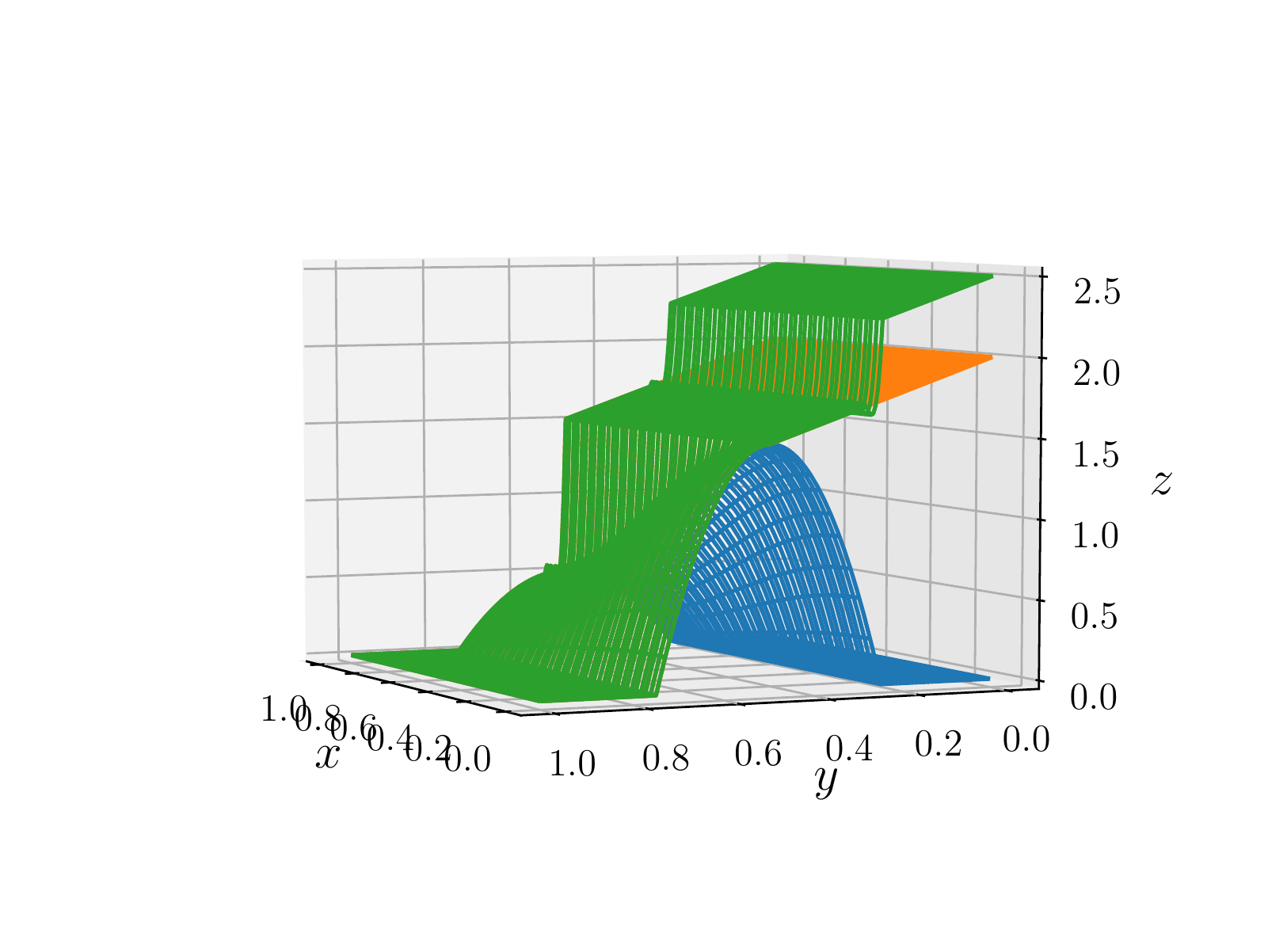}
      \caption{$t = 0$.}\label{subfig-seawater2d-a0}
    \end{subfigure}
    ~
    \begin{subfigure}[h!]{0.48\textwidth}
      \centering
      \includegraphics[width=\textwidth]{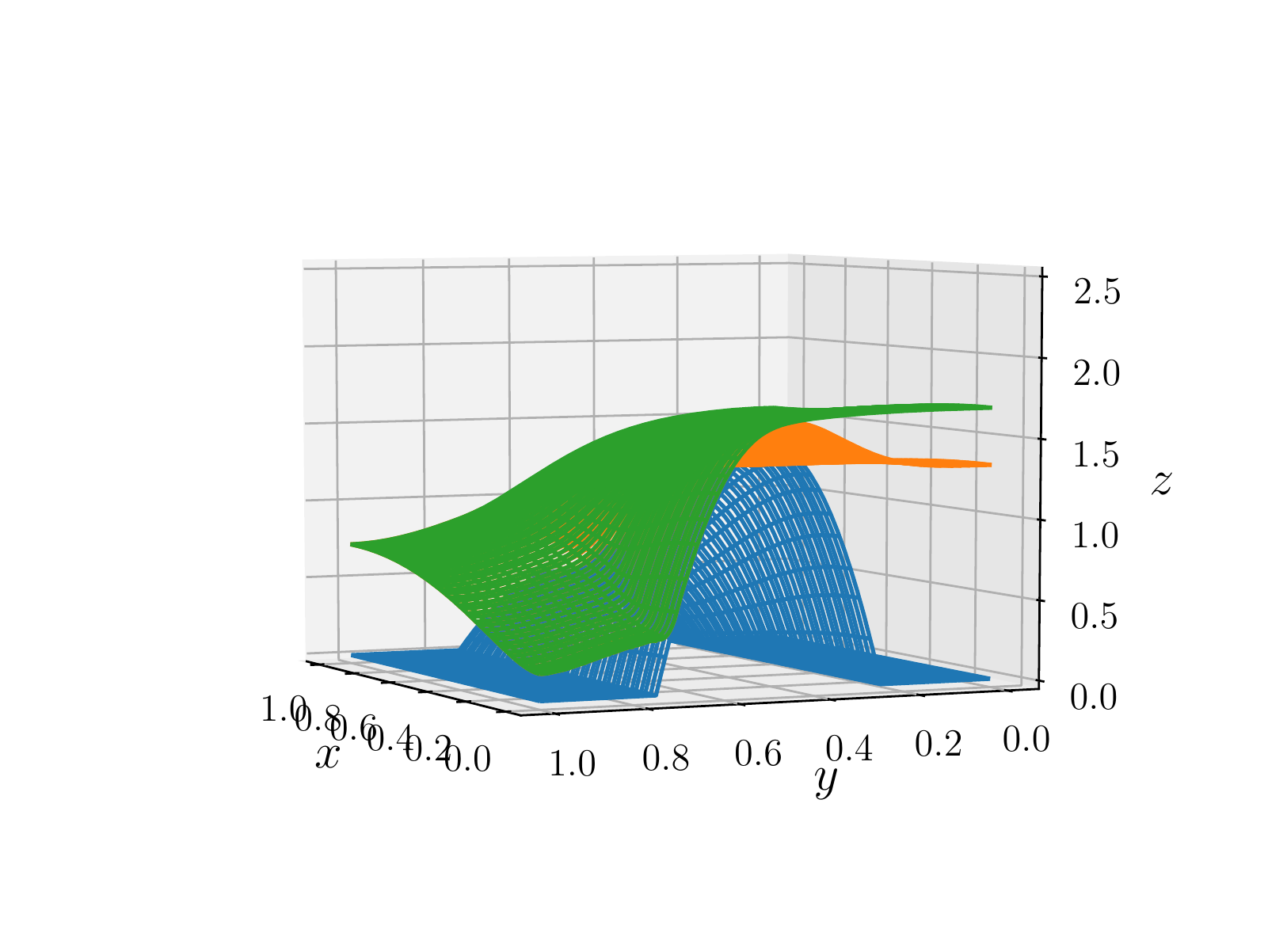}
      \caption{$t = 0.2$.}\label{subfig-seawater2d-a1}
    \end{subfigure}
    ~
    \begin{subfigure}[h!]{0.48\textwidth}
      \centering
      \includegraphics[width=\textwidth]{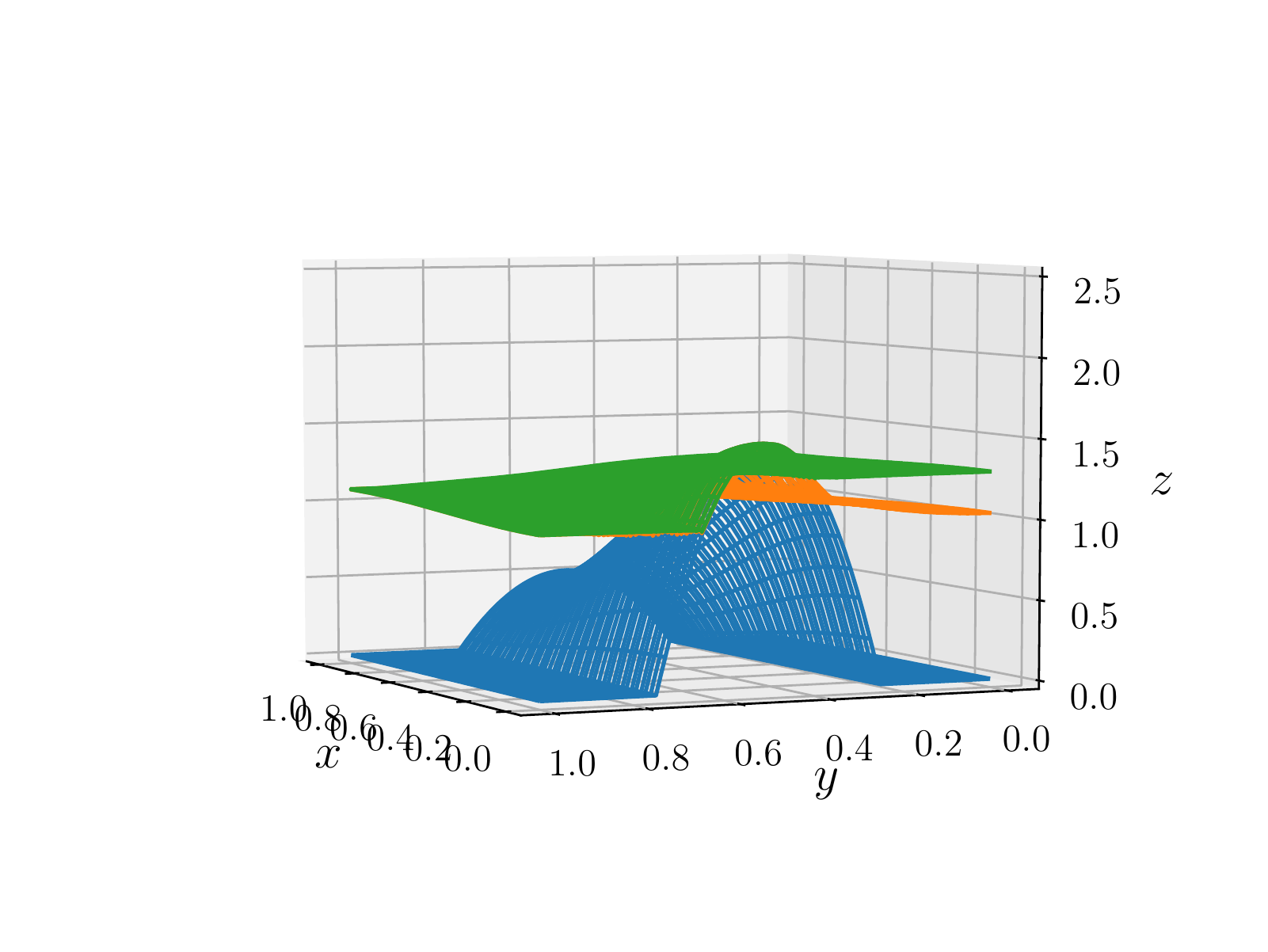}
      \caption{$t = 0.79$.}\label{subfig-seawater2d-a2}
    \end{subfigure}
    ~
    \begin{subfigure}[h!]{0.48\textwidth}
      \centering
      \includegraphics[width=\textwidth]{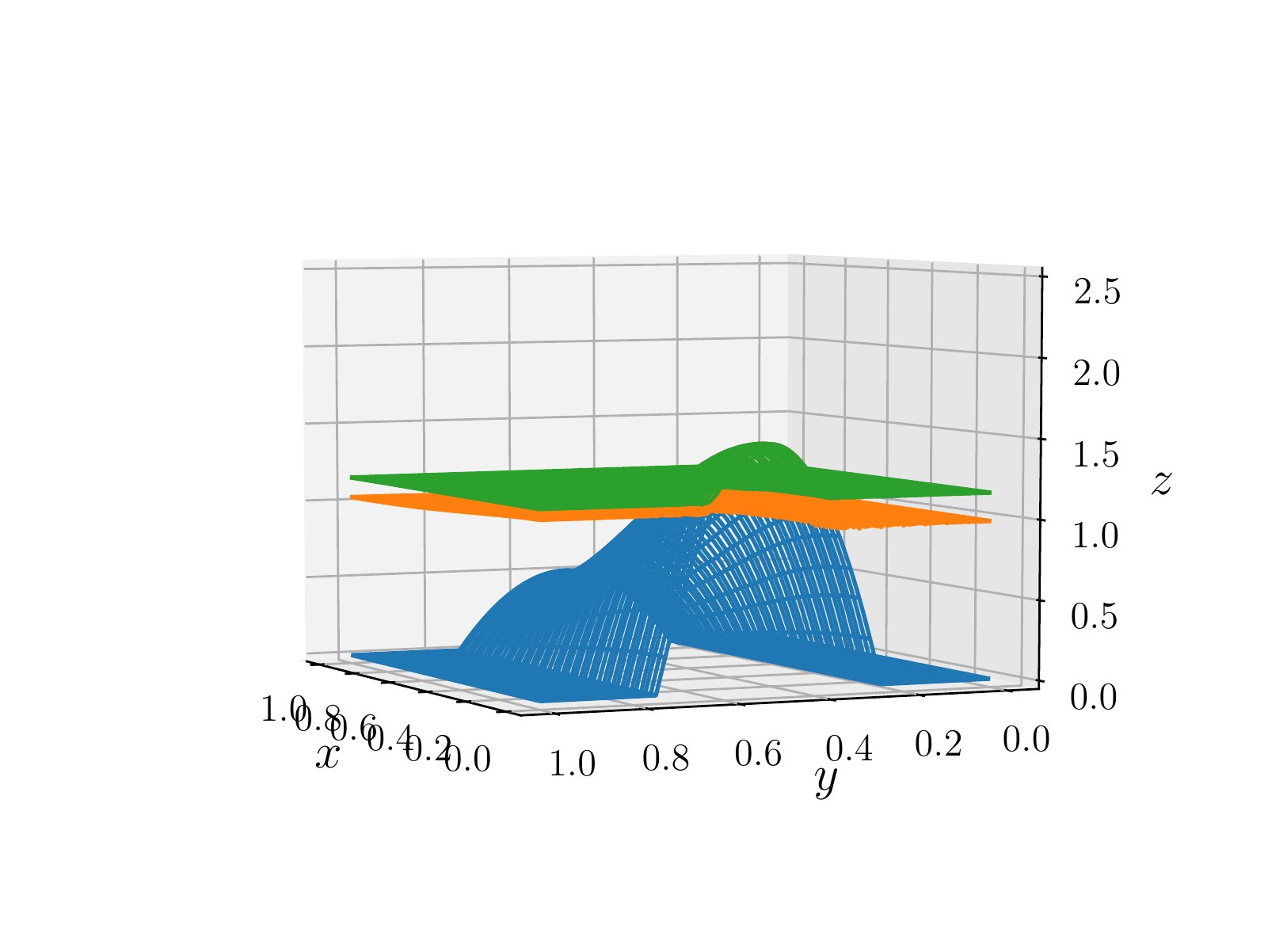}
      \caption{$t = 12$.}\label{subfig-seawater2d-a3}
    \end{subfigure}
    ~
    \begin{subfigure}[h!]{0.48\textwidth}
      \centering
      \includegraphics[width=\textwidth]{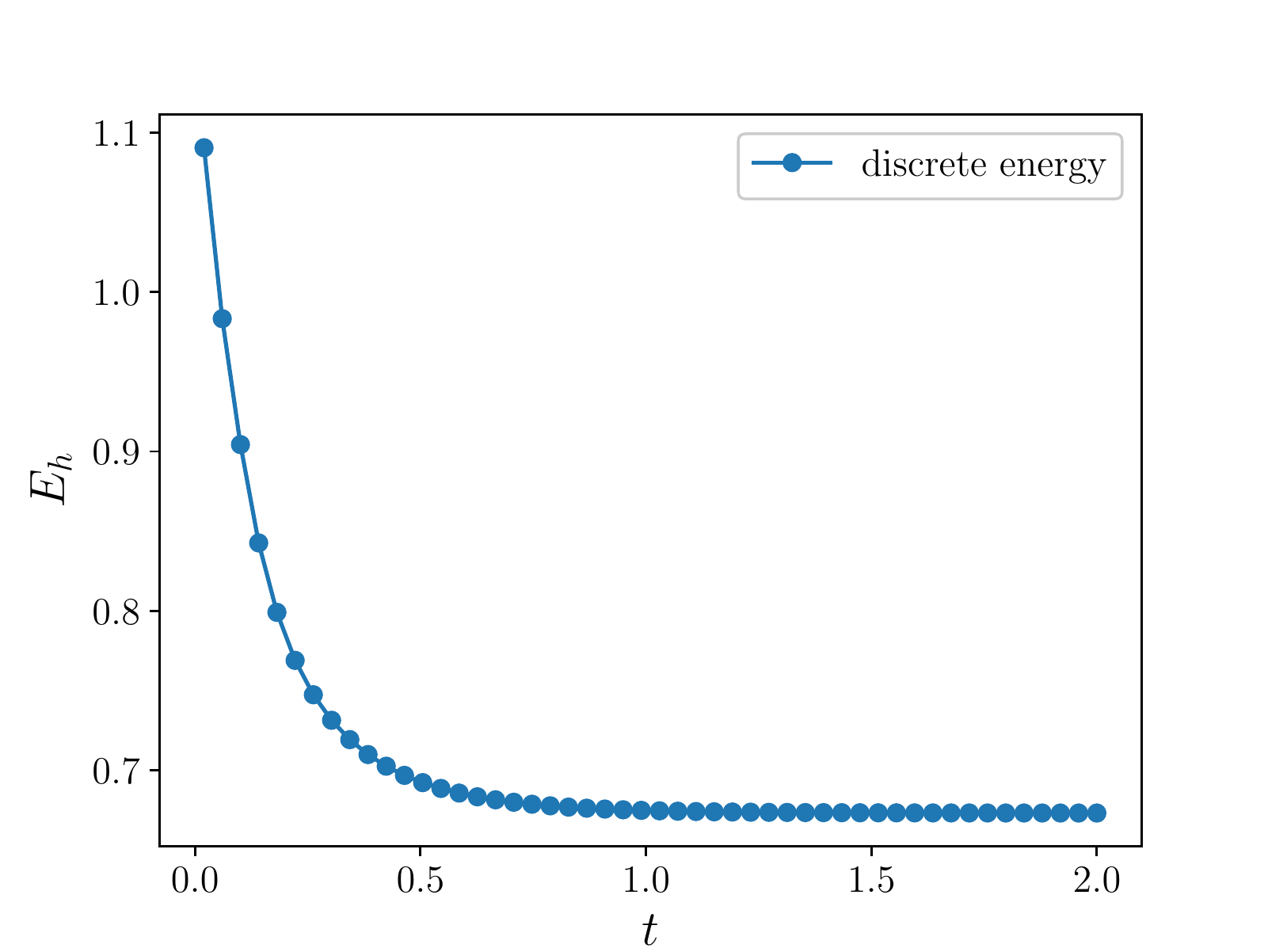}
      \caption{$E_h$.}\label{subfig-seawater2d-ent}
    \end{subfigure}
    \caption{Numerical solutions to the seawater intrusion problem in Example 
      \ref{examp:seawater} at $t = 0, 0.2, 0.79, 12$. Solutions are obtained with
      piecewise cubic polynomials on a uniform square mesh on
      $[0,1]\times [0,1]$, with $h^x = h^y
      =0.05$. The time step is set as $\tau = 0.002(h^x)^2$ in the 
    simulation. $b$, $b+\rho_2$ and $b+\rho_1+\rho_2$ are depicted in blue,
  orange and green respectively.}\label{fig-seawater}
\end{figure}
\end{Examp}

\section{Concluding remarks}\label{sec-conclusions}
\setcounter{equation}{0}
\setcounter{figure}{0}
\setcounter{table}{0}
In this paper, we extend the DG method in \cite{sun2018discontinuous} to solve cross-diffusion systems with 
a gradient flow structure. 
The difficulty is that the non-negativity of solutions is usually an
essential part for the entropy stability. One needs to
design numerical schemes to preserve the entropy structure, as well as 
to ensure the
positivity of solutions. We adopt the Gauss-Lobatto quadrature rule in the DG
method, so that the resulting semi-discrete scheme are subject to an entropy
inequality consistent with that of the continuum system. Furthermore, for a
class of problems, with a suitable choice of numerical fluxes, the numerical
method is compatible with the positivity-preserving procedure established in
\cite{zhang2010maximum} and \cite{zhang2017positivity}. The extension to two-dimensional 
problems on Cartesian meshes is also discussed. In our numerical tests, we
observe that the methods achieve high-order accuracy. Though the convergence
rate is reduced with odd-order piecewise polynomials, the optimal rate can be
retrieved by imposing larger Lax-Friedrichs constant in the numerical fluxes. 
Numerical simulations to problems with positivity-preserving issues have also
been performed.

\appendix
\section{Fully discrete entropy inequality}\label{sec-prooffully}

\begin{THM}
  Consider the Euler forward time discretization of the one-dimensional scheme
  \eqref{eq-1dscheme} with central and Lax-Friedrichs fluxes \eqref{eq-1dflux}
  on a uniform mesh.
  Suppose $\bz\cdot F\bz \geq \beta_F |\bz|^2$, $\bz\cdot D\bxi \bz \geq
  \beta_{D\bxi} |\bz|^2$,
  $|F|\leq \beta^F$ and $|D\bxi|\leq \beta^{D\bxi}$ uniformly in $\brho$ for
  some fixed constants 
  $\beta_F$, $\beta_{D\bxi}$, $\beta^F$ and $\beta^{D\bxi}$. 
  If $\tau \leq \min_{\substack{i}} (\frac{\beta_Fh^2}{4c_0^2c_{inv}^2
   (\beta^F)^2\beta^{D\bxi}},
   \frac{\beta_{D\bxi}h}{2c_0c_{inv}\beta^{D\bxi}\alpha_{i+\hf}})$, then 
 \[\frac{E_h^{n+1}-E_h^n}{\tau} \leq -\frac{\beta_F}{2}
   \nintO |\bu_h|^2 dx -
 \frac{\beta_{D\bxi}}{4} \sum_{i=1}^N\alpha_{i+\hf}
 |[\brho_h]_{i+\hf}|^2.\]
 Here $|F|$ and $|D\bxi|$ are $l^2$ matrix norms of $F$ and $D\bxi$
 respectively, $c_0$ is a constant for norm equivalence and $c_{inv}$ is the constant in the inverse estimate.
\end{THM}
\begin{proof}
  We omit all subscripts $h$ and superscripts $n$ in this proof.
  \begin{align*}
    {E^{n+1}-E} &= \sum_{i=1}^N\ninti
    \left(e(\brho^{n+1})-e(\brho)\right) dx \\
    &= \sum_{i=1}^N \ninti
    \left({e(\brho^{n+1})-e(\brho)-(\brho^{n+1}-\brho)\cdot
    \bxi}\right) dx+ \sum_{i=1}^N \ninti(\brho^{n+1}-\brho)\cdot\bxi dx\\
    &= \frac{1}{2}\sum_{i=1}^N \ninti (\brho^{n+1}-\brho)\cdot D\bxi(\bzeta)
    (\brho^{n+1}-\brho) dx + \sum_{i=1}^N
    \ninti(\brho^{n+1}-\brho)\cdot\bxi dx\\
    &:= A + B.
  \end{align*}
  Here we have applied the mean value theorem and $\bzeta$ is on the line
  segment between $\brho^{n+1}$ and $\brho$. 
  Let $\bseta = \frac{\brho^{n+1}-\brho}{\tau}$. Then
  \[A \leq \frac{1}{2}\max_{\bzeta }|D\bxi(\bzeta)|\nintO
    |\brho^{n+1}-\brho|^2 dx \leq \frac{\tau^2 \beta^{D\bxi}}{2}\nintO|\bseta|^2 dx \, .
  \]
Using the scheme \eqref{eq-1dscheme}, we obtain 
  \begin{align*}
    B &= \tau \left( -\sum_{i=1}^N \ninti \bu \cdot F \bu dx
      -\frac{1}{2}\sum_{i=1}^N\alpha_{i+\hf}
    [\bxi]_{i+\hf}\cdot [\brho]_{i+\hf}\right)\\
  &= \tau \left(-\nintO \bu\cdot F \bu dx -
   \frac{1}{2} \sum_{i=1}^N\alpha_{i+\hf}[\brho]_{i+\hf}\cdot
    D\bxi(\bzeta_{i+\hf})[\brho]_{i+\hf}\right)\\
  &\leq -\tau \beta_F \nintO |\bu|^2dx -
  \tau\frac{\beta_{D\bxi}}{2}\sum_{i=1}^N\alpha_{i+\hf}|[\brho]|^2_{i+\hf}\,,
  \end{align*}
 where $\bzeta_{i+\hf}$ lies between $\brho_{i+\hf}^-$ and $\brho_{i+\hf}^+$. 
Our main task is to estimate $\nintO |\bseta|^2 dx$. 
  \begin{align*}
    \nintO |\bseta|^2 dx =& \sum_{i=1}^N \ninti
    \frac{\brho^{n+1}-\brho^n}{\tau} \cdot \bseta dx\\
    =& \sum_{i=1}^N \left(-\ninti F \bu \cdot \partial_x\bseta dx + (\wFbu\cdot\bseta)_{i+\hf}^- -
    (\wFbu \cdot\bseta)_{i-\hf}^+\right)\\
    =& -\sum_{i=1}^N \ninti F\bu \cdot \partial_x\bseta dx - \sum_{i=1}^N
    \wFbu_{i+\hf} \cdot
    [\bseta]_{i+\hf}\\
    \leq& \, c_1\sum_{i=1}^N\ninti |F\bu|^2 dx +
    \frac{1}{4c_1}\sum_{i=1}^N\ninti
    |\partial_x\bseta|^2 dx\\
    &+ 2c_2\left(\sum_{i=1}^N |\{F\bu\}|^2_{i+\hf} + \sum_{i=1}^N
    \frac{\alpha_i^2}{4}|[\brho]|_{i+\hf}^2\right) + \frac{1}{4c_2}\sum_{i=1}^N
    |[\bseta]|^2_{i+\hf}.
  \end{align*}
  Here $c_1$ and $c_2$ are constants to be determined. Note that $\ninti
  |\cdot|^2 dx$ defines a norm on $P^k(I_i)$. Using norm equivalence and inverse
  estimates, we have 
\[c_0^{-1} \inti |p|^2 dx \leq \ninti |p|^2 dx \leq c_0\inti |p|^2 dx,\qquad  \forall
p\in {P}^k(I_i).\]
\[\sum_{i=1}^N \ninti |\partial_x \bseta|^2 dx \leq c_0\sum_{i=1}^N \inti
  |\partial_x \bseta|^2
  dx \leq \frac{c_0c_{inv}}{h^2} \sum_{i=1}^N \inti |\bseta|^2 dx \leq
\frac{c_0^2c_{inv}}{h^2} \sum_{i=1}^N \ninti |\bseta|^2 dx,\]
and 
\[\sum_{i=1}^N [\bseta]_{i+\hf}^2 \leq \frac{c_{inv}}{h}\sum_{i=1}^N\inti |\bseta|^2 dx \leq
\frac{c_0c_{inv}}{h} \sum_{i=1}^N \ninti |\bseta|^2 dx,\]
where $c_{inv}$ is a constant not less than $1$ in the inverse estimates.
Taking $c_1 = \frac{c_0^2c_{inv}}{h^2}$ and $c_2 = \frac{c_0c_{inv}}{h}$, then
we have 
   \[\nintO|\bseta|^2dx \leq \frac{2c_0^2c_{inv}}{h^2}
     \nintO|F\bu|^2 dx
     +
   \frac{4c_0c_{inv}}{h}\left(\sum_{i=1}^N|\{F\bu\}|^2_{i+\hf}+\sum_{i=1}^N\frac{\alpha^2_{i+\hf}}{4}|[\brho]|^2_{i+\hf}\right).\]
   Since
   \begin{equation*}
     \begin{aligned}
       \sum_{i=1}^N |\{F\bu\}_{i+\hf}|^2 \leq \frac{1}{2}\sum_{i=1}^N\left(|(F\bu)_{i-\hf}^+|^2 +
     |(F\bu)_{i+\hf}^-|^2\right)&\leq \frac{c_{inv}}{2h}\sum_{i=1}^N \inti
     |\cI(F\bu)|^2
   dx\\
   &\leq \frac{c_0c_{inv}}{2h}\sum_{i=1}^N \ninti |F\bu|^2 dx,
 \end{aligned}
 \end{equation*}
 one can obtain
 \begin{align*}
   \nintO|\bseta|^2dx \leq& \frac{4c_0^2c_{inv}^2}{h^2}
   \nintO|F\bu|^2 dx
   +\frac{c_0c_{inv}}{h}\sum_{i=1}^N{\alpha^2_{i+\hf}}|[\brho]|^2_{i+\hf}\\
   \leq& \frac{4c_0^2c_{inv}^2(\beta^F)^2}{h^2}\nintO|\bu|^2 dx
   +\frac{c_0c_{inv}}{h}\sum_{i=1}^N{\alpha^2_{i+\hf}}|[\brho]|^2_{i+\hf}
 .
 \end{align*}
 Therefore, we conclude that
 \begin{equation*}
   \begin{aligned}
     \frac{E^{n+1}-E}{\tau} 
     \leq& -\left(\beta_F-2c_0^2c_{inv}^2 (\beta^F)^2\beta^{D\bxi}
   \frac{\tau}{h^2}\right)\nintO |\bu|^2 dx \\
   &-\frac{1}{2}\sum_{i=1}^N\left({\beta_{D\bxi}}-c_0c_{inv}\beta^{D\bxi}\alpha_{i+\hf}\frac{\tau}{h}\right)\alpha_{i+\hf}
 |[\brho]_{i+\hf}|^2.
 \end{aligned}
 \end{equation*}
 The proof is completed by substituting the time step restriction into the
 inequality. 
\end{proof}

%
%
%
%

\section*{Acknowledgments}
JAC was partially supported by the EPSRC grant number EP/P031587/1. 
JAC would like to thank the Department of Applied Mathematics at 
Brown University for their kind hospitality and for the support 
through the IBM Visiting Professorship scheme. 
CWS was partially supported by ARO grant W911NF-16-1-0103 and 
NSF grant DMS-1719410.

\bibliographystyle{abbrv}

\end{document}